\documentclass[final]{siamltex}
\usepackage{amsmath}
\usepackage{mathtools}
\usepackage{mathrsfs}
\usepackage{amssymb}
\usepackage{amsfonts}
\usepackage{amsxtra}
\usepackage{amstext}
\usepackage{amsbsy}
\usepackage{amscd}
\usepackage{graphicx}
\usepackage{float}
\usepackage{cite}
\usepackage{xcolor}
\usepackage{srcltx} %Inverse search
\usepackage{marginnote,slashbox}
\usepackage{rotating}
\definecolor{darkblue}{rgb}{0.0,0.0,0.6}
\definecolor{darkgreen}{rgb}{0.0,0.6,0.0}
\usepackage[pdftex,colorlinks=true,urlcolor=darkblue,citecolor=darkblue,linkcolor=darkblue]{hyperref}
\usepackage[utf8]{inputenc}      % for German umlauts
\usepackage{booktabs}            % to get fancier tables
\usepackage{url}                 % to typeset URLs
\usepackage[T1]{fontenc}         % to enable correct hyphenation of words with umlauts
\usepackage{algorithmic}         % defines algorithmic environment
\usepackage[pagewise]{lineno}
\usepackage{calc}
\usepackage[notcite,notref]{showkeys}

\numberwithin{table}{section}    % for Table 1.1
\numberwithin{figure}{section}   % for Figure 1.1
\numberwithin{equation}{section} % for equation (1.1)
\usepackage{my_latex_commands}

\newtheorem{assumption}[theorem]{Assumption}
\newtheorem{remark}[theorem]{Remark}

\begin{document}

\title{ a novel way of computing the shape derivative for a class of non-smooth pDEs and
 its impact on deriving 
 necessary conditions for locally optimal shapes}
%\title{Strong Stationarity for a highly nonsmooth optimal control problem  with  control  constraints
%}
\date{\today}
\author{L.\,Betz\footnotemark[1]}
\renewcommand{\thefootnote}{\fnsymbol{footnote}}
\footnotetext[1]{Faculty of Mathematics,  University of W\"urzburg,  Germany}
\renewcommand{\thefootnote}{\arabic{footnote}}

\maketitle

\begin{abstract}
We derive  necessary  conditions for locally optimal shapes of a design problem governed by a non-smooth PDE. The main particularity of the state system  is the lack of differentiability of the nonlinearity.  We work  in the framework  of the functional variational approach (FVA) \cite{pen}, which has the capacity to transfer geometric optimization problems into optimal control problems, the set of admissible shapes being parametrized by a large class of continuous mappings. In the FVA setting, we   introduce a sensitivity analysis   technique that is novel even for  smooth PDEs. We emphasize that we do not resort  to extensions on the hold-all domain or any kind of approximation of the original PDE.  The computation of the directional derivative of the state w.r.t.\,functional variations results in a new way of computing the shape derivative. The presented approach allows us to  handle  in the objective pointwise observation  and derivatives of the state on an observation set   as well as  distributed observation terms.
In addition, we introduce  the  concept of locally optimal shapes and we put into evidence its connection to locally minimizers of the corresponding control problem. With directional differentiability results for the control-to-state map at our disposal, we can  then  state necessary  conditions for locally optimal shapes in general non-smooth settings.
%\ko{ it opens the door to ...and allows for extsniosn to other nons-mm prob. . we show the rel between fva and lsm, as well as connections to the speed method.}
\end{abstract}



\section{Introduction}

Shape and topology optimization problems have been extensively addressed  in the past decades by different types of  methods, see e.g.\,the classical contributions \cite{sz,pir,hp,dz,top_book, allaire}. Given a hold-all domain $D$, an optimal design problem, defined
over a given family $\OO$ of subdomains of $D$   may have 
the following structure
\begin{equation}\begin{aligned}
\min_{\Omega\in\mathcal{O}}&\  \ \JJ(\O,y_\O)
\\\text{s.t.}&\ \ 
 y_\Omega \text{ solves }(Q_\Omega),
 \end{aligned}
\end{equation}
where   $\JJ$ is a given shape functional, and $(Q_\Omega)$ is the state system (typically a PDE or a VI) posed in $\O$ with solution $y_\O$ (the state). 
As one is interested in finding information about the optimal shape,  there is a zoo of works  addressing   the sensitivity analysis of  $\O\mapsto \JJ(\O,y_\O)$ see e.g.\,\cite{ nsz, hs, sok, k_sturm,kun,kk, kov_k, kov_k0, zs} and the references therein. The most common   differentiability notions appearing in all these  papers are   based on variations of the geometry. 

%\ro{\cite{kov_k0} uses paramters, it is  a special type of so pb}

% common concept are the shape derivative \cite{sz} and the topological derivative \cite{top_book}. Both notions are   based on variations of the geometry.  

A meanwhile well-established strategy for  handling  shape and topology optimization problems 
%both at the theoretical and numerical level 
is the so-called functional variational approach (FVA), which was  first introduced in  \cite{mnt}. This technique will be employed in the present paper too. One of the essential aspects about the FVA is that shape perturbations do not have a prescribed geometric form, as it is usual in the literature. Hence, there is no reason to work with differentiability concepts   such as 
 shape/topology derivative \cite{sz,top_book}.
This methodology has  proven its viability in numerous contributions that deal with various optimal design problems \cite{pen, t_jde, to,  hmt, mt_new2, tm_lin, kirchhoff, stokes, plate, oc_t}.
 We  refer here also  to the   recents works \cite{p1,p3} where convex  approximation schemes and  optimality systems for optimal design problems governed by non-smooth PDEs were provided. 
As a possible drawback of the FVA we mention 
its
  limitation to dimension two (which stems from the fact that it uses a  Hamiltonian technique
based on the Poincar\'e-Bendixson theory \cite[Ch.\,10]{hsd}). We point to  {\cite{it_h}} where  some hints
on the possible relaxation of the dimensional restriction are given.

In the context of the FVA, the boundaries of the domains appear as zero level sets of certain shape functions   defined over a larger fixed domain. This idea  is mostly associated with the prominent  level set method (LSM) \cite{os, 3}.
%, see also \cite{lsm_vi} for LSMs for VIs. 
However, the last-mentioned approach  is   concerned with  the time evolution of  curves and/or surfaces and we note that LSMs are mainly employed for numerical purposes; cf.\,also \cite{1,2} for rigorous mathematical investigations. 
While our admissible shapes  shall arise as  sublevel sets as well, the FVA   is essentially different from the LSM, as it   is strongly based on an optimal control framework, where the level set function is the control.

The FVA 
allows one to switch from the shape optimization problem to a control problem in a function space. The advantage of the latter is that it is more amenable to the application of tools from optimal control theory.
%The FVA is strongly based on an optimal control framework, where the shape optimization problem  is rewritten  as  an  optimal control problem where the level function plays the role of the control.
The new  admissible set   consists of    shape functions  (also termed level functions or parametrizations). These belong   to a certain family   $\FF$ that has the capacity to generate all admissible domains $\O \in \OO$ \cite[Prop.\,2.8]{p1}, \cite[Thm.\,4.1 and 4.2]{dz}.

If $\OO$ is the set of all non-empty $\CC^2$ subdomains of $D\subset \R^2$ whose boundaries do not intersect $\partial D$, this transfer  takes place via   the relation
\[\O=\O_g:=\{x \in D: g(x)< 0\}, \quad g \in \FF,\]where $\FF$ is  the set of shape functions defined as \begin{equation}\label{f}\FF:=\{g \in C^{2}(\bar D):|\nabla g(x)|+|g(x)|>0 \ \forall\,x\in D,\ g(x)>0 \ \forall\,x\in\partial D,\  g^{-1}((-\infty,0]) \neq \emptyset\}. \end{equation}
Note that, given $g \in \FF$, the $\CC^2$ set $\Omega_g$ may have many components (their number is finite by \cite[Prop.\,2]{pen}), but it is always possible to redefine $g$ so that $\O_g$ is a domain \cite[Lem.\,2.9]{p1} (in section \ref{sec:so} we select a component  by asking that $E \subset \O_g$, where $E$ is  an observation subdomain). However, the boundary of this particular domain $\O_g$ may be a union of several closed curves, that is, $\O_g$ may have holes. Hence, the FVA is related to topological optimization too and we want to underline this aspect.

The \textit{main aim of this paper} is to introduce  a   direct approach for   differentiating solutions of \textit{non-smooth}  PDEs  with respect to perturbations of $\CC^2$ planar domains in the context of the FVA. We focus on the following Dirichlet problem:
 \begin{equation}\label{eq0} \begin{aligned}
-\laplace y + \beta(y  ) &=f   \quad \text{in }\O ,\\\quad y  &= 0 \quad \text{on }\partial \O ,\end{aligned}\end{equation}where $\beta$ is  a not necessarily differentiable function, cf.\,Assumption \ref{assu:stand} for details. 
%As already mentioned, we work in two dimensions only and $D\subset \R^ 2$ is a   fixed bounded hold-all domain. Moreover, the perturbed domains are supposed to be of class $\CC^2$ while $f\in L^p(D),\ p>2$, which ensures the maximal regularity of the state $y \in W^{2,p}(\O)\cap W^{0,1}(\O) $. 
As already mentioned,  $D\subset \R^ 2$ is a   fixed bounded hold-all domain. For the  right-hand side of the PDE we assume $f\in L^p(D),\ p>2$, which ensures the maximal regularity of the state $y \in W^{2,p}(\O)\cap H_0^1(\O) $, since $\O \in \CC^2$. 

Given $g \in \FF$, we consider so-called \textit{functional variations} $g+\l h$, $\l>0$ ``small'', where $h\in \FF$ is fixed \cite{mnt}.   That is,  the perturbation of the domain $\O=\O_g$  is $\O_{g+\l h}$. We are thus concerned with  the investigation of 
\begin{equation}\label{diff_s}\lim_{\l \searrow 0}\frac{\SS(g+\l h)-\SS(g)}{\l},\end{equation}
where $\SS:g  \mapsto y_g $ is the solution operator of \eqref{eq0} with $\O=\O_g$. 
{Because of the non-smooth  structure of  \eqref{eq0}, $\SS$ is not expected to be G\^ateaux-differentiable and we remark here that the sensitivity analysis of non-smooth problems on perturbed domains has been addressed in the literature  only for VIs  so far  \cite{sign_shape, sign_shape0,sz, sok,hs, hl, nsz}. 

 We point out that, to the best of our knowledge, the procedure  in the present manuscript  is entirely novel even for smooth problems. A less complex situation for $\CC^\infty$ domains in arbitrary dimension was discussed in  \cite[Sec.\,3]{2}. There  the connection between the shape derivative of $\JJ$ and the derivative of $g\mapsto \JJ(\O_g)$ was established, but only for shape functionals that \textit{do not} depend on the state. We emphasize that once we compute and analyse \eqref{diff_s}, we are in the position to calculate the directional derivative of $g\mapsto \JJ(\O_g, y_g)$ and  since we deal with a non-smooth PDE, the word "directional" is essential here. 
 
 When examining \eqref{diff_s}, the first question that will arise is how does the directional derivative of $\SS$  look like. It turns out that the limit in \eqref{diff_s} yields in fact   the shape derivative of $y_{\O_g}$, see Remark \ref{rem:sd}. 
Another fundamental question is  in which space does the convergence   \eqref{diff_s} hold,   the challenge here being that  $\SS(g)$ maps to   $ W^{2,p}(\O_g)\cap H_0^1(\O_g) $, which changes when perturbing the domain. Simply extending   $\SS(g)$ by zero is not   a solution, as the  $W^{2,p}$ regularity of the state gets lost outside $\O_g$. We underline that this   regularity is  important as it ensures that $\nabla y_g$ is continuous and the boundary condition in the 'linearized' PDE \eqref{q}-\eqref{q0} is well defined.
To tackle these issues we shall inspect  the convergence of the difference quotient in \eqref{diff_s} mainly in $\O_g \cap \O_{g+\l h}$ and in particular on the boundary of $\O_g \cap \O_{g+\l h}$, since  it turns out that this dominates the   convergence in the interior of the common domain (Proposition \ref{prop:lr}). Our key tools to begin with are the representation of $\CC^2$ planar closed curves via Hamiltonian systems \cite{pen}, their approximating behaviour \cite{p3} and differentiability properties \cite{oc_t}. When we go deep in the analysis we see that an essential point is to observe that a certain part of the boundary of $\O_g \cap \O_{g+\l h}$ remains fixed with varying $\l$, see \eqref{b_ol} and figure \ref{fig:ol}. We emphasize that the $W^{2,p}$ regularity of the state as well as the stability of the boundary perturbations $\partial \O_{g+\l h}$ play  at different stages of the manuscript   fundamental roles.
%\\key point:observing that a certain part of the boundary $\O \cap \O_{g+\l h}$ remains fixed with varying $\l$

The first two main results of the manuscript are Theorems \ref{thm:h1} and \ref{thm:lr}, where we establish 
two different types of differentiability  results for $\SS$. These will allow us to consider in the objective pointwise observation  and derivatives of the state on an observation set (due to Theorem  \ref{thm:h1}, cf.\,subsection \ref{sec:gobs}) as well as  distributed observation terms (thanks to Theorem \ref{thm:lr}, see subsection \ref{sec:do}).

%the system in variations corresponds to the one that chracterized the shape derivative; we point out the connection between fva and the speed method as well as LSM

A second key contribution of this paper is the introduction of the notion of \textit{locally optimal domain} and the derivation of necessary optimality conditions therefor. 
%To the best of our knowledge,  only globally optimal shapes have been investigated in the literature so far.
This   concept shall be presented in the framework of  the optimal design problem
\begin{equation}\tag{\ensuremath{P_\O}}\label{p_sh0}
 \left.
 \begin{aligned}
  \min_{\O \in \OO, E \subset \subset \O} \quad &\JJ(\O, y_\O)\\     \text{s.t.} \quad & 
   \eqref{eq0}\end{aligned}
 \quad \right\}
\end{equation}where   the observation set $E\subset \subset D$ is a non-empty subdomain of $D$. The problem \eqref{p_sh0} can be rewritten as an optimal control problem in function space, which in its reduced from reads
\begin{equation}\label{min_j}\min_{g \in \FF, g<0 \text{ in }\bar E} \JJ(\O_g,y_g).\end{equation}
We show that for  each locally optimal shape $\O=\O_g$  of \eqref{p_sh0}, all associated  parametrizations $g\in \FF$ are local minimisers  of \eqref{min_j} in a certain sense (Theorem \ref{prop:loc}).  We note that the admissible set in \eqref{min_j} is not  necessarily convex. To obtain that the directional derivative of the reduced objective $j:g\mapsto \JJ(\O_g,y_g)$ is non-negative in all feasible directions $h$, we only need that $\{g \in \FF:g<0 \text{ in }\bar E\}$ is open, which is indeed the case. The derivation of necessary optimality conditions in primal form for locally optimal shapes (in the context of non-smooth PDEs)  constitutes a second main finding of the present manuscript (Theorem \ref{prop:necc}).
%We not necessarilydo not need that $\FF_E$ is convex. Since $\FF_E$ is open, we are allowed to test with functional variations. 
%In addition, findi $j'(g;h)$
%\ko{This    opens the door for many possiblieties:finding descent direction in a gradient method; in the case of smooth fct...optimality systems }

%\ro{vorne:there are may approaches,some of them do not even look at the shape deriv of the dontrol to state mao..kun, sturm}

% \begin{equation}\label{eq0} 
%-\laplace y_g  + \beta(y_g ) =f   \quad \text{in }\O_g,\quad y_g = 0 \quad \text{on }\partial \O_g,\end{equation}
%To (*) we can associate the solution operator 
%\[g \in \FF \mapsto y_g \in H_0^1(\O_g) \cap H^2(\O_g).\]we emphasize that all $\CC^2$ domains are covered.

We end the introduction by putting    into evidence the differences between our technique  and other comparable works. We start with the discrepancies with other papers  in the context of the FVA. Though this particular method has often been used   for differentiation purposes, the novelty in the present paper is that one does not alter the original PDE; we recall that our idea is to look at the sensitivity of the solution operator of  \eqref{eq0} (with $\O=\O_g$) with respect to shape functions only, cf.\,\eqref{diff_s}. The  existing contributions on the topic address only smooth problems and may be divided into two categories. A large part of them resorts to extensions   on the hold-all domain by means of a regularization of the Heaviside function  \cite{stokes, plate, kirchhoff, mt_new2}, which yields an approximating PDE on the fixed domain. Other works \cite{hmt, t_jde,to} are based on the  so-called Hamiltonian approach, which is more direct. Therein, an additional control variable defined on $D$ is introduced so that the PDE is posed in the hold-all domain as well. However, the boundary condition appears  as a state-control constraint   and  the constrained optimal  control problem is in fact equivalent to the original shape optimization problem \cite[Cor.\,3.1]{t_jde}. That is, even though the extended PDE is obviously not the same as the original one (defined in $\O_g$),  at the minimization level, the method used in  \cite{hmt, t_jde,to} is exact.

We now turn our attention to the rest of the literature on  the sensitivity analysis of problems on perturbed domains with emphasis on non-smooth problems \cite{sign_shape, sign_shape0,sok,hs, hl, nsz} and 
 we see that  one of the most common approaches is the prominent speed/velocity method \cite{sz}. As opposed to the FVA, this has the advantage that it works in arbitrary dimensions too.
 The typical idea  is to introduce   a  smooth transformation of the domain, generated by a family of vector fields, which associates to each point in $\O$  a unique point in the perturbed shape and viceversa. In this manner, the   perturbed  smooth domains are of the type $T_\l(\O)$, where $T_\l$ is a bijection satisfying certain conditions; as a special case we mention the so-called {perturbation of identity} where the variations are generated via $T_\l(\O)=(\mathbb{I}+\l \theta) \Omega$, where $\theta$ is  a smooth vector field (the velocity field). The speed method allows one to rewrite the perturbed equation in the original unperturbed domain   $\O$. Deriving w.r.t.\,$\l$  then yields the so-called material derivative, by means of which the shape derivative of the state can be computed. This procedure is explained and applied at length in \cite{sz} and has been used by numerous authors. We refer here in particular to \cite[Sec.\,5.3, Thm.\,5.3.1]{hp}, where the shape derivative for weak formulations of PDEs is calculated by   requiring $\O$ to be only measurable, bounded and open. However, we underline that \cite{hp} highly relies on  the implicit function theorem so that  the results cannot be  expected to hold for non-smooth problems. Another remarkable aspect about the aforementioned contributions  is that, since they involve the computation of the so-called material derivative, they require the given right-hand side, in our case $f$, to posses weak derivatives (mostly $f \in H^1(D)$). By contrast, we do not transfer the perturbed PDE to the original domain $\O=\O_g$, but   we examine \eqref{diff_s} on the common domain $\O_ g \cap \O_{g+\l h}$, see subsection \ref{sec:dd}. That is, we do not compute material derivatives, so that $f\in L^p(D)$ is sufficient in the present manuscript and we want to stress this matter.

Though our functional  variational concept may seem different from the classical purely geometrical one \cite{sz}, we will see that, 
%  the limit in \eqref{diff_s} corresponds  in fact to the  shape derivative \cite{sz}. 
we are in fact proposing a novel way of computing the shape derivative, cf.\,section \ref{sec:sm} and Remark \ref{rem:sd}. While this is restricted to $
\CC^2$ planar domains, it has the advantage that the involved data $f$ and $\beta$ do not require smoothness assumptions (as already pointed out).
It turns out that the boundary of our perturbed domain $\O_{g+\l h}$ arises as  $T_\l(\partial \O_g)$, where $T_\l$ is the flow associated to  a certain family of vector fields $\theta \in \CC([0,\l_0),\CC_c^{0,1}(D;\R^2))$ for which $\theta(0)=W$, see Definition \ref{def:w}.  This   (initial) velocity field $W$    can be put in connection with a linearization of a Hamiltonian system which describes the boundary of the original domain, cf.\,\eqref{w}.
 %At the end of the manuscript, we will look at the parallel between the FVA and the velocity approach \cite{sz} and we will see why this observation should not be surprising, as the family of vector fields  in \cite{sz} has a counterpart in the FVA, see Definition \ref{def:w}. It turns out that the boundary of our perturbed domain $\O_{g+\l h}$ arises as  $T_\l(\partial \O_g)$, where $T_\l$ is the flow associated to the a family of vector fields $\theta \in \CC([0,\l_0),\CC_c^{0,1}(D;\R^2))$ for which $\theta(0)=W$. 

Finally, let us   state an essential matter regarding  the different perspective offered by the FVA when compared with the LSM.
In the LSM setting one employs the Hamilton Jacobi equation   to update the level set function in numerical algorithms. This describes the evolution of 
a mapping $g_\l$ with respect to $\l$ (which there plays the role of the time variable). As  a result of perturbing the zero level set of $g_0=g$ by a known vector field $\theta:\{g=0\} \to \R^2$, the Hamilton Jacobi equation reduces to finding $g_\l$ so that
\begin{equation}\label{hj}\frac{d}{d\l}g_\l(x)+\theta (x)\nabla g_\l(x)=0 \quad  \text{in }D,\quad  g_0=g.\end{equation}
The FVA theory offers an "other way around" point of view: we already know the perturbations of the functions, i.e., $g_\l=g+\l h$, and, as we will see, we can compute the velocity  field $\theta=W$ associated to the  variation of the zero level set of $g$, cf.\,Definition \ref{def:w}. Indeed, this satisfies \eqref{hj} with $g_\l=g+\l h$ on $\{g=0\}$, as shown by \eqref{eq:W} below. 
%This velocity field is introduced in  Definition \ref{def:w} and it plays a fundamental role when connecting \eqref{diff_s} with the shape derivative of $\O \mapsto y_\O$ obtained via the speed method, cf.\,section \ref{sec:sm}.
%\ko{indeed setting $g_\l:=...$ in \eqref{hj} yiels an equation satisfied by W.}

\subsection*{Outline of the paper} The present work is organized as follows. After introducing the notation, we recall in Section \ref{sec:p} some essential previous findings  concerning the description of $\CC^2$ planar curves via Hamiltonian systems and other fundamental properties. In section \ref{sec:3} we state the precise assumptions on the fixed  data in \eqref{eq0} after which we examine the continuity properties of the state with respect to domain perturbations (subsection \ref{sec:cont_S}). Subsection \ref{sec:dd} is entirely focused on the computation of \eqref{diff_s}. Here we introduce the candidate for the directional derivative of $\SS$ and  we investigate 
\eqref{diff_s} on the common domain $\O_ g \cap \O_{g+\l h}$ as well as on its boundary. This shall result in the first main contribution of the paper, namely Theorem \ref{thm:h1}. In Proposition \ref{prop:outside} we gather information about  \eqref{diff_s} outside $\O_ g \cap \O_{g+\l h}$,  which yields   the second main finding of this work, i.e., Theorem \ref{thm:lr}. Starting with section \ref{sec:so} we focus on the optimal design problem \eqref{p_sh0} and its equivalent formulation as a control problem. The main novelty here is due to  subsection \ref{sec:lo}, where we introduce the concept  of locally optimal shape and its relation with local optimality for control problems such as \eqref{min_j}. Section \ref{sec:noc} is devoted to the derivation of necessary optimality conditions for locally optimal shapes. These are stated in a general framework in Theorem \ref{prop:necc}.   The main differentiability findings from  section \ref{sec:3} then allow  us to apply Theorem \ref{prop:necc} for 
 objectives with  pointwise observation  and derivatives of the state on an observation set (subsection \ref{sec:gobs}) as well as  distributed observation terms (subsection \ref{sec:do}). The final section \ref{sec:sm} of the manuscript is devoted to a comparison between the FVA and the speed/velocity method, and may be seen as a confirmation of the differentiability  results obtained in section \ref{sec:3}.

%\ro{if the convergence in lr norm not okay or in the h1 norm, we still have the convergence in supremum norm,sec 3.3, which was shown with convexity assumptions, in a different manner}

%\ro{we took the convergence in supremum norm out, because the fact that $x_{obs} \in \partial \O$ does not pose challenges when J$=....\underbrace{y(x_{obs})}_{=0}$; but if $q_\l \to q$ in $W^{1,p}(\o)$ were not available, then we would need to use the convergence in supremum norm,sec 3.3., which was shown with convexity assumptions}

\subsection*{Notation}
Throughout this manuscript, $\AA \subset \subset D$ means that the closure of the set $\AA$ is a compact subset of $D$ and we use the symbol 
$|\cdot|$ for the Euclidean norm. In addition, $\mu(\AA)$ is the measure of a Lebesgue measurable set $\AA\subset \R^2$ and $\CC_c(\AA)$ is the set of all continuous functions on $\AA\subset \R^2$ with compact support. We use the short-hand notation  $\{g=0\}$ for the zero level set  $\{x\in D:g(x)=0\}$ and if $\{g=0\}$ is a union of closed curves, $\{g=0\}^{\circ}$ denotes  a component of this union (a single periodic curve). The mapping $\dist_\AA:\R^2 \to \R^+$ associates to a point its distance to  a compact set $\AA\subset \R^2$. The    Hausdorff-Pompeiu  distance between two  compact sets $\AA,\BB\subset \R^2$   is defined as 
\begin{equation}\begin{aligned}
d_{\HH}(\AA , \BB):=\max \{\max_{x \in \AA}\dist(x,\BB),\max_{\widetilde x \in \BB}\dist(\widetilde x, \AA)\}.\end{aligned}\end{equation}
Given a normed space $X$, the 
closed ball centered at $0\in X$ with radius $M>0$ is denoted by $\overline {B_X(0,M)}$.

\section{Preliminaries}\label{sec:p}In this section, we collect some (meanwhile well-established) key findings about the  description of planar curves via Hamiltonian systems \cite{pen}, their approximating behaviour \cite{p3} and differentiability properties \cite{oc_t}. 
%  \begin{equation}\label{delta}|\nabla g|\geq \delta_g \text{ on }\partial \O_g.\end{equation}
To this end, we recall the definition of our set of admissible shape functions: \begin{equation}\label{eq:f}\FF=\{g \in \CC^{2}(\bar D):|\nabla g(x)|+|g(x)|>0 \ \forall\,x\in D,\  g(x)>0 \ \forall\,x\in\partial D, \ g^{-1}((-\infty,0]) \neq \emptyset\}. \end{equation}
We begin with a fundamental result that stands at the core of the FVA:
\begin{proposition}[{\cite[Prop.\,2]{pen}}]\label{prop:ham}
For each $g \in \FF$, the level set $\{g=0\}$  is a finite union of disjoint closed $\CC^2$ curves, without self intersections and not intersecting $\partial D$. Each component of $\{g=0\}$ is parametrized by the solution $z :[0,\infty)\to \{g=0\}^\circ$ of  the Hamiltonian system \begin{equation}\label{ham}
 \left\{
 \begin{aligned}
({z_1})'(t) &= -\partial_2 g(z(t)), 
    \\   (z_2)'(t) &= \partial_1 g(z (t)),
   \\ {z}(0) &= x_0, \quad  t \in [0,\infty)\end{aligned}
 \quad \right.
\end{equation} when some initial point  $x_0 \in \{g=0\}^\circ$ is chosen.\end{proposition}
\begin{remark}Note that $z:[0,T)\to \{g=0\}^\circ$ is a bijection, where $T>0$ stands for the period of the closed curve $\{g=0\}^\circ$. If $|\nabla g(\cdot)|=1$ on $\{g=0\}^\circ$, its period is given by %solving the fixed point problem
%\[l=\int_0^T |\nabla g(z(s))|\,ds\]
 %\ko{(usually, $T>0$ is chosen as 
 the length of the curve, see e.g.\,\cite{pressley}.
\end{remark}

Let $g \in \FF$. Given  a second   function $h \in \CC^2(\bar D)$ (not necessarily in $\FF$), we call {functional variations}
{\cite{mnt}} the perturbations $g + \lambda h$, $\lambda \in \mathbb{R}$  ``small''. The next result goes to show that the preceding Proposition \ref{prop:ham} remains unaffected by  small enough perturbations of the original shape function $g$.

\begin{lemma}[{\cite[Prop.\,3.4]{p3}}]\label{lem:glh}
For each  $g \in \FF$ and $h \in \CC^2(\bar D)$, there exists $\l_0>0$ so that for all $\l \in (0,\l_0)$, we have $g+\l h \in \FF$ as well. That is,
 the assertion of Proposition \ref{prop:ham} stays true for the level set $\{g+\l h=0\}$ and each of its components is parametrized by the solution of  the Hamiltonian system \begin{equation}\label{ham_e}
 \left\{
 \begin{aligned}
   z_{\l,1} '(t)&=- {\partial_2  (g+\l  h)}  (z_\l (t)) ,
    \\    z_{\l,2} '(t)&={\partial_1 (g+\l  h)} (z_\l(t)) ,
   \\   z_\l(0)&=x^\l_0 ,\quad  t \in [0,\infty) \end{aligned}
 \quad \right.
\end{equation} when some initial point $x^\l_0 \in D$ is chosen on each $\{g+\l h=0\}^\circ$. \end{lemma}
\begin{proof}The assertion follows by the same arguments used in the proof of \cite[Lem.\,3.1]{p3}. As a consequence of $g \in \FF$ combined with  the compactness of $\bar D$,
we have by Weierstrass theorem that
\begin{equation}\label{gd}|\nabla g|+|g| \geq  \delta \quad \text{in }\bar D\end{equation}for some $\delta>0.$ Since $g+\l h \to g$ in $\CC^2(\bar D)$ as $\l \searrow 0$, it follows that there exists $\l_0>0$, depending only on $\delta>0$ and the given data,  so that the first condition in \eqref{eq:f} is satisfied by $g+\l h$ for all $\l \in (0,\l_0)$. A similar argument  yields that the second inequality in \eqref{eq:f} is fulfilled as well. Finally, we have by \eqref{gd} and $g^{-1}((-\infty, 0])\neq \emptyset$ that there exists a point $\tilde x \in D$ so that $g(\tilde x)<0$. Since $g+\l h$ converges uniformly towards $  g$ we get  that $(g+\l h)(\tilde x)<0$ for $\l>0$ small enough dependent only on $\tilde x$, and thus on $g$. This allows us  to deduce the desired statement.
\end{proof}

It turns  out that the finite number of components of $\{g+\l h=0\}$ is independent of $\l$, and, in fact, equal to the number of components of $\{g=0\}$, provided that $\l$ is small enough. This is a consequence of the following.
  
%\ko{ \begin{lemma}\label{lem:l0}
%For each $\lambda>0$ there exists $\e_0>0$ so that
%\[\{x \in \bar D\setminus E: \gbe(x) \in [-\e,\e]\} \subset V_{\lambda} \quad \forall\,\e\in (0,\e_0],\]
%where \[V_{\lambda}:=\{x \in \bar D\setminus E:  \dist(x,\partial \O_{g=0})<\lambda\}.\]
%\end{lemma}}
\begin{proposition}[{\cite[Lem.\,3.7, 3.8]{p3}}, Uniqueness of the approximating curve]\label{lem:unique}
Let $g \in \FF$ and $h \in \CC^2(\bar D)$. 
Given $\{g=0\}^\circ$,  there exists $\e>0$ and  $\l_0(\e)>0$ so that for each $\l \in (0,\l_0(\e)]$ there is exactly one closed curve $\{g+\l h=0\}^\circ$ satisfying
\[\{g+\l h=0\}^\circ \subset \{x \in  D:  \dist(x,\{g=0\}^\circ)< \e\} .\]  
%\ko{In particular, it holds
 %\begin{equation}\label{grad_u}
%|\nabla (g+\l h)(x)| \geq  \delta/2 \quad \forall\,x \in  D\text{ with }  \dist(x,\{g=0\}^\circ)< \e, \quad \forall\,\l\in (0,\l_0(\e)],\end{equation}where $\delta>0$ is dependent only on $D$. Thus, for each $\l \in (0,\l_0(\e)],$ the mapping $g+\l h $  cannot have local extremum points in $\{x \in  D:  \dist(x,\{g=0\}^\circ)< \e\}$.}
Moreover,  we have the convergence 
\begin{equation}\label{h_d_c}
d_{\HH}(\{ g+\l h=0\} , \{g=0\})\to 0 \quad \text{as }\l \searrow 0,\end{equation}
where $d_{\HH}(\{ g+\l h=0\} , \{g=0\})$ denotes the Hausdorff-Pompeiu  distance between the compact sets $\{ g+\l h=0\}$ and $ \{g=0\}$, i.e.,
\begin{equation}\begin{aligned}
&d_{\HH}(\{ g+\l h=0\} , \{g=0\})\\ &\quad:=\max \{\max_{x \in \{g=0\}}\dist(x,\{ g+\l h=0\}),\max_{\widetilde x \in \{ g+\l h=0\}}\dist(\widetilde x, \{g=0\})\}.\end{aligned}\end{equation}
\end{proposition}

From now on, whenever we write 'approximating curve' of $\{g=0\}^\circ$ we refer to the unique curve $\{g+\l h=0\}^\circ$  from Proposition \ref{lem:unique} associated to it.
\begin{proposition}[{ \cite[Prop.\,2.6]{to}, \cite[Prop.\,3.10]{p3}}]\label{lem:vec}
Fix $g \in \FF$ and $h \in \CC^2(\bar D)$. Let $\{g=0\}^{\circ}$ be a closed curve with  periodicity  $T$ and let $\{g+\l h=0\}^{\circ}$ be its approximating curve with periodicity $T_{\l}$.
The following assertions are true:
\begin{itemize}
\item[(i)]  If $x_0^\l \to x_0,$ then, for each $B>0$, it holds  $   \| z_\l - z\|_{C^1([0,B];\R^2)} \to 0 $ \text{as }$\l \searrow 0,$ where $z_\l $   and $  z$ are  the trajectories of the Hamiltonian systems \eqref{ham_e} and \eqref{ham} corresponding  to $\{g+\l h=0\}^{\circ}$ and $\{g=0\}^{\circ}$, respectively;
\item[(ii)] $T_{\l} \to T$ as $\l \searrow 0.$
\end{itemize}
\end{proposition}

To be able to discuss the differentiability properties of the Hamiltonian system \eqref{ham} w.r.t.\,$\l$, we need to require that the zero level set of $h$ intersects each of the components of $\{g=0\}$ in at least one point (which we choose as initial condition in the Hamiltonian system corresponding to $\{g=0\}^\circ$), cf.\,figure \ref{fig:init}.  \begin{figure}[ht]
\begin{center}
\includegraphics[width=0.5\textwidth]{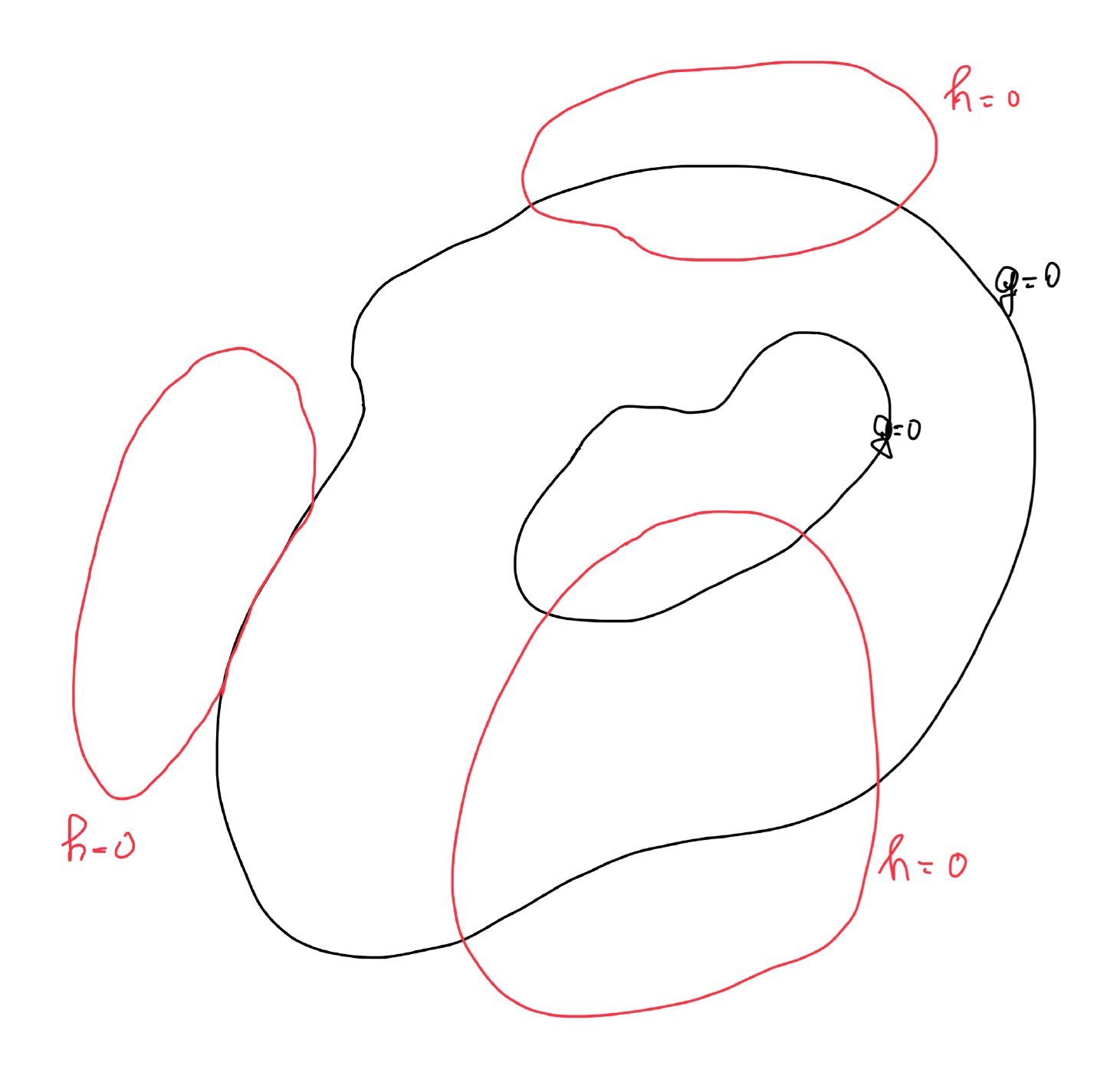}
\end{center}
\caption{The zero level set of $h$ intersects each component  of $\{g=0\}$.}
\label{fig:init}
\end{figure}   That is, 
\begin{equation}\label{hx00}\{h=0\}\cap \{g=0\}^\circ \neq \emptyset \quad \forall\,\{g=0\}^\circ\subset \{g=0\},\end{equation}and 
in \eqref{ham} we will   choose the initial condition $z(0)=x_0\in \{g=0\}$, so that 
 \begin{equation}\label{hx0}h(x_0) = 0.\end{equation}This mild restriction is due to the lack of differentiability of the operator $\l \mapsto \proj_{\{g+\l h=0\}}x_0$ \cite[p.\,9]{oc_t}. Note that the class of perturbed domains $\O_{g+\l h}$ is still large: it covers all non-empty $\CC^2$ subdomains of $D$  whose boundaries intersect the boundary of the  original domain $\O_g$, see also figure \ref{fig:ol}. 

\begin{proposition}[{\cite[Prop.\,3.2]{oc_t}}]\label{prop:diff}
In addition to $g \in \FF$, suppose that  $h \in \CC^2(\bar D)$ satisfies \eqref{hx00}. Let $z_\l $   and $  z$ be  the trajectories of the Hamiltonian systems \eqref{ham_e} and \eqref{ham} (both with initial point $x_0$, cf.\,\eqref{hx0}) corresponding  to $\{g+\l h=0\}^{\circ}$ and $\{g=0\}^{\circ}$, respectively. 
 Then,  it holds
\begin{equation}\label{wa}
w_\l:=\frac{z_\l - z}{\l} \to w \text{ in } C^1([0,T];\R^2) \quad \text{as }\l \searrow 0,\end{equation}where $T >0$ is the period of $\{g=0\}^{\circ}$ and  $w= (w_1,w_2)$ is the unique solution of 
\begin{equation}\label{w}
 \left\{
 \begin{aligned}
(w_1)'(t)& =  -\nabla\partial_2 g(z(t))\cdot w(t)
-\partial_2 h(z(t)),
\\
(w_2)'(t)& =  \nabla\partial_1 g(z(t))\cdot w(t)
+\partial_1 h(z(t)),\\
w(0)& = 0, \quad t \in [0,T].
\end{aligned}
 \quad \right.
\end{equation} 
\end{proposition}

\section{The Dirichlet problem}\label{sec:3}This section is dedicated to the study of the differentiability properties of the solution of \eqref{eq0} w.r.t.\,perturbations of the two-dimensional domain.   These perturbations take place via functional variations as described in the introduction.

In the sequel, $g \in \FF$ is fixed and $\O_g=\{x\in  D:g<0\}$ is assumed to be a (possibly multiply connected) planar \textit{domain} (we recall that $\O_g \in \CC^2$, cf.\,Proposition \ref{prop:ham}, see also \cite[Thm.\,4.2]{dz}). Note that, since, in general, $\O_g$ may have more than one component, it is reasonable to 'select' one of them in the context of solving PDEs. This will be done in section \ref{sec:so}  via the condition $E \subset \O_g$ and for the moment we just assume that $\O_g$ is connected (however, it may have holes).
We keep the whole time in mind the  properties:
\[\partial \og=\{x\in  D:g=0\},\]
 \[\overline{\og}^c=\{x\in  D:g>0\},\]where $\overline \og^c$ is the complement of the set $\overline \O_g$. Note that these  may be found for instance in \cite[Thm.\,4.2]{dz}.

We are interested in  the following non-smooth problem:
\begin{equation}\label{eq}\begin{aligned}
-\laplace y_g  + \beta(y_g ) &=f   \quad \text{in }\O_g,\\ y_g &= 0 \quad \text{on }\partial \O_g,\end{aligned}\end{equation}where $f\in L^p(D),\ p>2$ and $\beta$ is  a not necessarily differentiable function. 
%\ko{The homogeneous boundary condition can be replaced by $y=\psi$, with $ \psi \in W^{2-1/p,p}(\partial \O_g)$ (alternatively one might take $ \psi \in W^{2,p}(D)$, since $\partial \O_g$ is not always a priori known, e.g.\,in the context of optimal design problems) are fixed. 
%where, with a little abuse of notation, $\widetilde \psi$ is an extension (not unique) of $\widetilde \psi$ on the whole domain D.\\
%We choose to work with \eqref{eq0} in order to keep the non-smmoth problem as simple as possible, while focusing on the challenges raised by ..........}

The next requirement on the non-linearity $\beta$ is a standiing assumption and it is thus supposed to hold throughout the manuscript (without mentioning it every time).
  \begin{assumption}[The non-smoothness]\label{assu:stand}
 \begin{enumerate}
 \item\label{it:stand2}The  function  $\beta: \R \to \R$ is   
monotone increasing and {locally} Lipschitz continuous {in the following sense: F}or all $M>0$, there exists 
  a constant $L_M>0$ such that
  \begin{equation*}
   |\beta(y_1) - \beta(y_2)| \leq L_M |y_1 - y_2| \quad \forall\, y_1, y_2 \in [-M,M] .
  \end{equation*}
  \item\label{it:stand3}The mapping $\beta$ is directionally differentiable at every point, i.e., 
  \begin{equation*}
   \Big|\frac{\beta(y + \tau \,\dy) - \beta(y)}{\tau} - \beta'(y;\dy)\Big| \stackrel{\tau \searrow 0}{\longrightarrow} 0 \quad \forall \, y,\dy \in \R.
  \end{equation*}
  \end{enumerate}
\end{assumption}

By Assumption \ref{assu:stand}.\ref{it:stand2}, it is {straightforward} to see that the Nemytskii operator $\beta:L^\infty(\AA) \to L^\infty(\AA)$ is well-defined for each measurable  set $\AA\subset \R^2$. Moreover, this is Lipschitz continuous on bounded sets, i.e., for every  $M > 0$, 
  there exists $L_M > 0$ so that
  \begin{equation}\label{eq:flip}
   \|\beta(y_1) - \beta(y_2)\|_{L^r(\AA)} \leq L_M \, \|y_1 - y_2\|_{L^r(\AA)} \quad \forall\, y_1,y_2 \in \clos{B_{L^\infty(\AA)}(0,M)},\ \forall\, 1\leq r \leq \infty.
  \end{equation}In addition,
  $\beta:L^\infty(\AA) \to L^\varrho(\AA)$, $1\leq \varrho<\infty$, is directionally differentiable, as a result of Assumption \ref{assu:stand} and the Lebesgue's dominated convergence theorem, that is,
    \begin{equation}\label{eq:dir}
   \Big\|\frac{\beta(y + \tau \,\dy) - \beta(y)}{\tau} - \beta'(y;\dy)\Big\|_{L^\varrho(\AA)} \stackrel{\tau \searrow 0}{\longrightarrow} 0 \quad \forall \, y,\dy \in L^\infty(\AA), \ \forall\, 1\leq \varrho<\infty.
  \end{equation}
\begin{definition}[The control-to-state map associated to \eqref{eq}]\label{S}
We define
\begin{equation}
\SS:g \in \FF \mapsto y_g \in  H^1_0(\O_g) \cap W^{2,p}(\O_g),\end{equation}
where $y_g$ solves the equation \eqref{eq} on the domain $\O_g$.\end{definition}

\begin{remark}\label{rem:S}
According to Definition \ref{S},  $\SS(g)$ exists only as an element of $H_0^1(\O_g).$ Whenever we write $\SS( g)$ as an element of $H_0^1(D)$ in what follows,  we think of its extension by zero outside $\O_g$.
\end{remark}

The principal aim of this section is to look at the differentiability properties of $\SS$  w.r.t.\,$g\in \FF$ (in direction $h\in \CC^2(\bar D)$), that is, we are interested in calculating
\[\lim_{\l \searrow 0}\frac{\SS(g+\l h)-\SS(g)}{\l}\]in a suitable space.
%Note that, since $\beta$ is not necessarily differentiable, we cannot expect, in general, the operator $\SS$ to be differen.
%In addition to $g\in \FF$, the function $h\in \CC^2(\bar D)$ is fixed  as well in what follows and it is supposed to satisfy \eqref{hx00}, that is, the zero level set of $h$ should have at least one common point with each closed curve contained in $\{g=0\}$; this  ensures the viability of Proposition \ref{prop:diff} throughout the paper.

 For simplicity, we will use in the sequel the abbreviations $y_g:=\SS(g)$, $y_{g+\l h}:=\SS({g+\l h})$, where $\l>0$ is small enough so that the unicity of the approximating curve in Proposition \ref{lem:unique} is guaranteed.

\subsection{Continuity properties of $\SS$}\label{sec:cont_S}
In this subsection we  look at the convergence behaviour of $y_{g+\l h}$ as $\l \searrow 0$.
We start with some essential observations.

\begin{lemma}[{\cite[Thm.\,8.33]{gt}}]
Let $g \in \FF$ and $h\in \CC^2(\bar D)$. Then, there exists $\l_0>0$ so that  
\begin{equation}\label{w2p_bound}\|y_{g+\l h}\|_{\CC^{1,\alpha}(\bar \O_{g+\l h})}\leq c,\quad \forall\,\l\in (0,\l_0],
\end{equation}where $c>0$ is independent of $\l$ and $\alpha>0$ is such that $W^{2,p}(\O_{g+\l h})\embed \CC^{1,\alpha}(\bar \O_{g+\l h})$.
\end{lemma}\begin{proof}First, we observe that 
\begin{equation}\label{cl}\|y_{g+\l h}\|_{\CC(\bar \O_{g+\l h})}\leq c \quad \forall\l \in[0,\l_0),\end{equation}where  $\l_0>0$ is small (as in Lemma \ref{lem:glh}) and where $c>0$ is independent of $\l$, see e.g.\,the proof of {\cite[Thm.\,4.5]{troe}}. Then, since \eqref{eq} associated to $g+\l h$ can be rewritten as
\[-\laplace y_{g+\l h}=-\beta(y_{g+\l h})+f \ \ \text{ in }\O_{g+\l h},\ y_{g+\l h}|_{\partial \O_{g+\l h}}=0, \]we obtain by \cite[Thm.\,8.33]{gt} (see also the remark at the end of Sec.\,8.11 in \cite{gt}) the estimate
\[\|y_{g+\l h}\|_{\CC^{1,\alpha}(\bar \O_{g+\l h})}\leq C(\|y_{g+\l h}\|_{\CC(\bar \O_{g+\l h})}+\|\beta(y_{g+\l h})-f\|_{L^p(\O_{g+\l h})}),\]where $\alpha>0$ is such that $W^{2,p}(\O_{g+\l h})\embed \CC^{1,\alpha}(\bar \O_{g+\l h})$ and $C>0$ is independent of $\l$. 

Note that the latter assertion is due to the fact that $C>0$ depends only on the $\CC^{1,\alpha}$ norms of the $\CC^2$ mappings which define the local representation  of $\partial \O_{g+\l h}$ (see the comments following \cite[Thm.\,8.33]{gt} and the proofs of \cite[Thm.\,6.1, 6.6]{gt}). These local representations can be chosen as in \cite[proof of Thm.\,4.2]{dz}:
\[\Psi_x:\UU(x)\to \R^2,\quad \Psi_x(y):=\frac{1}{|\nabla g_\l(x)|}\Big((y-x)^T(-\partial_2 g_\l(x),\partial_1 g_\l(x)), {g_\l(y)}\Big),\]
where $\UU(x)$ is a neighbourhood of $x\in \partial \O_{g+\l h}$ and we abbreviate $g_\l:=g+\l h$.
Hence, the fact that $C>0$ is independent of $\l$ reduces to the uniform boundedness of $\|g_\l\|_{\CC^{1,\alpha}},$ which is true for $\l$ small (even when we take the $\CC^2$ norm).

 Thanks to \eqref{cl}, \eqref{eq:flip} and $\O_{g+\l h} \subset D$ ($\l$ small), we finally deduce \eqref{w2p_bound} from the above estimate.\end{proof}

\begin{lemma}\label{lem:l}Let $g \in \FF$ and $h\in \CC^2(\bar D)$.
For each $v \in \CC_c(\O_g)$ there exists $\l(v,g)>0$ so that $v \in \CC_c(\O_{g+\l h})$ for all $\l\in [0, \l(v,g))$.
\end{lemma}\begin{proof}We show that each compact $K$ subset of $\O_g$ is a compact subset of $\O_{g+\l h}$. Indeed, 
since $g$ is continuous, we have
\[g(x)\leq - \delta \quad \forall\, x \in K \subset \subset \Omega_g\]
for some $\delta>0.$
Thus, there exists $\l>0$, small, independent of $x$, so that
\[g(x)+\l h (x) \leq  -\delta/2<0 \quad  \forall\, x \in K.\]Thus, $K \subset \Omega_{g+\l h}$.
 For a fixed $v \in \CC_c( \O_{g})$ there exists a compact subset $K$ of $\O_g$ so that 
$v \in \CC_c(  K)$ and hence, $v \in \CC_c (  \O_{g+\l h})$ for all $\l\in [0, \l(v,g))$.
\end{proof}

\begin{remark}
An analogous argument to the one from the proof of  Lemma \ref{lem:l} shows that each compact subset of $\bar \O_g^c$ is a compact subset of $\bar \O_{g+\l h}^c$ for $\l$ small enough and thus, $\O_{g+\l h}$ converges to $\O_g$ in the sense of compact sets \cite[Def.\,2.2.21]{hp}. In fact,
\[d_\HH(\bar \O_{g+\l h},\bar \O_g)\to 0,\]
\[d_\HH(\bar D \setminus \O_{g+\l h},\bar D \setminus \O_g)\to 0,\]
\begin{equation}\label{mu}\mu(\O_{g+\l h}) \to \mu(\O_g),\end{equation}
see for instance \cite[p.\,63]{hp} and \cite[Appendix 3]{nst_book}.
%we have h conv for open sets,in the sense of charact fct, usual hc 
\end{remark}

%This plays a central role in this section.
%It shows  how adabvantageous the fva can be , reliability and simplicity of the fva in the context of dirichlet problems . Not only  it allows us to use the same set of space functions, but even the volume integrals are built over the original unperturbed domain. For neumann problems we have a similar situation, however in that case we canot take integrals over $\O_g$ in the perturbed problem, but at least the test set remains unchanged , which is the main difficulty.

One of the key characteristics of the FVA is that $\partial \O_g$ is divided by $\O_{g+\l h}$ into two disjoint (not necessarily connected) parts  that are independent of $\l$ (recall that, in view of \eqref{hx0}, we have $\partial \O_g \cap \partial \O_{g+\l h} \neq \emptyset$). Here we include the case when one of these disjoint parts is void (and the other is the whole boundary), that is, $\O_g \subset \O_{g+\l h}$. This observation is due to the following essential (easy to check) fact
 \[\partial \O_g \cap \partial \O_{g+\l h }=\{g=h=0\}.\]
Hence, we  have a representation of $\partial \O_g$ as  a disjoint union of two sets that are both independent of $\l$, namely
\begin{equation}\label{div}
\partial \O_g=\G_1 \cup \G_2
\end{equation}where the relatively open set $\G_1$ is defined as
\begin{equation}\label{g1}\G_1:=\partial \O_g \setminus \bar  \O_{g+\l h}\end{equation}and the relatively closed set $\G_2$ is given by 
\begin{equation}\label{g2}\G_2:=\partial \O_g \cap \bar \O_{g+\l h} .\end{equation} 
This fundamental remark is true regardless of the number of components of $\partial \O_g$ (i.e., even when $\O_g$ has  holes), as a consequence of Proposition \ref{lem:unique}. {To ensure that the zero level set of $h$ does not intersect $\{g=0\}$ an infinite number of times}, we require in the remainder of the paper that $h \in \FF$, which by Proposition \ref{prop:ham}, yields that $\{h=0\}$ is a finite union of disjoint closed $\CC^2$ curves. Then, locally, $\{h=0\}^\circ$ can intersect  $\{g=0\}$ in a single point or these two zero level sets may have common regions consisting of curves, as depicted   on the left side in figure \ref{fig:g2}.
 \begin{figure}[ht]
\begin{center}
\includegraphics[width=0.5\textwidth]{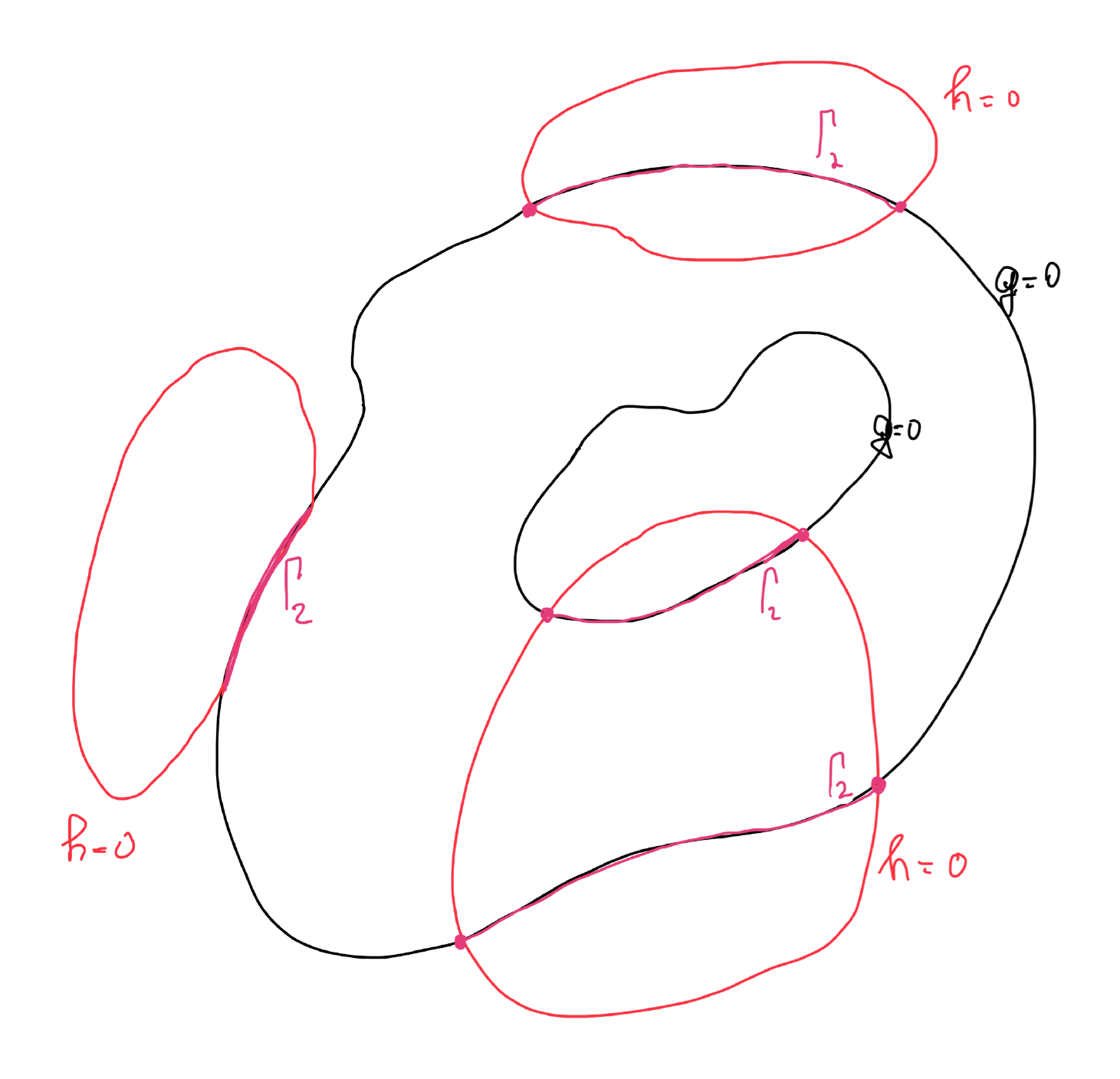}
\end{center}
\caption{The set $\G_2$ (the part of the boundary of $\partial \O_g$ included in $\bar \O_{g+\l h}$)}
\label{fig:g2}
\end{figure}  
%\begin{remark}{Let us already emphasize that  the  set $\G_2$ and the associated convergence \eqref{conv_g2} will play an  important role later in the context of establishing limits for difference quotients on the boundary, cf.\,Lemma \ref{lem:h0}. In the next subsection, we work on domains $\O_{g+\l h} \cap \O$ and $\G_2$ will be the non-varying part of the boundary of $\O_{g+\l h} \cap \O$, see \eqref{b_ol}. }\end{remark}
 \begin{proposition}\label{lem:conv_yl}Suppose that $g, h \in \FF$ and \eqref{hx00} is true. Then it holds  
\begin{equation}\label{conv_yl0}y_{g+\l h} \weakly y_g \text{ in }W_0^{1,\varrho}(D)\quad \text{as }\l \searrow 0,\quad \forall\varrho \in [1,\infty),
\end{equation}
%\begin{equation}\label{conv_yl0'}\laplace y_{g+\l h} \to \laplace y \text{ in }\CC_c^1(\O_g)^\star \quad \text{as }\l \searrow 0,\end{equation}
%\begin{equation}\label{conv_yl}y_{g+\l h} \to y \text{ in }\CC^{1,\alpha}(\bar \O_{\l_0} \cap \bar \O)\quad \text{as }\l \searrow 0\end{equation}for some $\alpha>0$ and some fixed $\l_0>0$, small.\ro{assuming orange lines come closer to boundary of $\O$ for $\l $ small; but this is in fact not so impoertant; imp is uniform conv on $\G_2$}
\begin{equation}\label{conv_g2}
    \nabla y_{g+\l h} \to  \nabla y_g \text{ in }\CC^{0,\alpha}(\G_2)\quad \text{as }\l \searrow 0\end{equation}for some $\alpha>0$.
\end{proposition}
\begin{proof}(i) \textit{Proof of \eqref{conv_yl0}}. 
From \eqref{w2p_bound} 
%and  $W^{2,p}(\O_{g+\l h})\embed W^{1,\varrho}(\O_{g+\l h})$ 
we have on a subsequence  
\begin{equation}\label{w1}y_{g+\l h} \weakly \widetilde y \text{ in }W_0^{1,\varrho}(D).\end{equation}
In view of Lemma \ref{lem:l}, we can test  the equation associated to $y_{g+\l h}$  with some arbitrary   $v \in \CC_c^\infty(\O_g)\subset \CC_c^\infty(\O_{g+\l h})$, where $\l=\l(v)$ is small.  We have 
\begin{equation}\label{e00}
\int_{\overline{\supp} v} \nabla y_{g+\l h} \nabla v\,dx+\int_{\overline{\supp} v} \beta(y_{g+\l h})v\,dx =\int_{\overline{\supp} v}f v \,dx, \quad \forall\, \l\in [0,\l(v)).  \end{equation} Passing to the limit $\l \searrow 0 $ in \eqref{e00}, where we rely on \eqref{w1}, the compact embedding $W_0^{1,\varrho}(D) \embed \embed L^{\infty}(D)$ and the continuity properties of $\beta$, see Assumption \ref{assu:stand}.\ref{it:stand2},
yields
\begin{equation}\label{e01}
\int_{\overline{\supp} v} \nabla \widetilde y \nabla v\,dx+\int_{\overline{\supp} v} \beta(\widetilde y)v\,dx =\int_{\overline{\supp} v}f v \,dx,  \end{equation}i.e.,
\begin{equation}\label{e0}
\int_{\O} \nabla \widetilde y \nabla v\,dx+\int_{\O} \beta(\widetilde y)v\,dx =\int_{\O}f v \,dx \quad \forall\, v \in   \CC_c^\infty(\O).\end{equation}
It remains to show $\widetilde y=0$ on $\partial \O$. 
Due to \eqref{w1} we have \[\g y_{g+\l h} \weakly \g \widetilde y \text{ in }W^{1-1/\varrho,\varrho}(\partial \O),\]where $\g:W^{1,\varrho}(D) \to W^{1-1/\varrho,\varrho}(\partial \O)$ is the trace operator. Thus,
$\widetilde y=0$ on $\G_1$, cf.\,\eqref{g1},
and \begin{equation}\label{w1'}\g y_{g+\l h} \weakly \g \widetilde y \text{ in }W^{1-1/\varrho,\varrho}(\G_2),\end{equation}where we recall the definition of $\G_2$ from \eqref{g2}.
  Fix $t \in [0,T]$ so that $z(t) \in \O_{g+\l h}$, where $z$ is the solution of the Hamiltonian system \eqref{ham}. Since $h\in \FF$, by assumption, we have that $\{h=0\}$ is a finite union of disjoint closed $\CC^2$ curves (see Proposition \ref{prop:ham}). Thus, $t$ belongs to a finite union of subintervals of $[0,T]$, say $I$. It holds 
\[|y_{g+\l h}(z(t))-y_{g+\l h}(z_\l(t))|\leq c |z(t)-z_\l(t)|^{\hat \alpha}, \quad t \in I,\]with some $\hat \alpha >0$ and $c>0$ independent of $\l$, where we used \eqref{w2p_bound};
 %and Morrey's estimate; 
 recall that we take  the initial point in both Hamiltonian systems \eqref{ham} and \eqref{ham_e} to be  $x_0$, cf.\,\eqref{hx0}. Now, in light of Proposition \ref{lem:vec}
we have
\[\sup_{t \in I}|y_{g+\l h}(z(t))|\to 0\] which yields
\[\gamma y_{g+\l h} \to 0 \quad\text{in } \CC(\G_2).\]Due to \eqref{w1'} we have $\widetilde y=0$ on $\G_2$ as well. In conclusion, the weak limit in \eqref{w1} belongs to $W_0^{1,\varrho}(\O)$ (recall here \eqref{div}) and since it satisfies \eqref{e0}, we obtain $\widetilde y=y_g$. Taking another subsequence as in \eqref{w1} yields the same limit and we arrive at the desired convergence result for the whole sequence $\{y_{g+\l h}\}$.
%\\\ko{ii) \textit{Proof of \eqref{conv_yl0'}}. Subtracting \eqref{e01} from \eqref{e00}, using $\widetilde y=y$ and \eqref{conv_yl0}, as well as Assumption \ref{assu:stand}, imply
%\[(\laplace(y_{g+\l h}-y),v)_{}=\int_{\O} \nabla ( y_{g+\l h}-y) \nabla v\,dx \leq L_M\|y_{g+\l h}-y\|_{L^\infty(\clos \supp v)}\|v\|_{L^1(\clos \supp v)}\quad \forall\, v \in \CC^1_c(\O).  \]
%it follows \[\laplace ( y_{g+\l h}-y) \to 0 \text{in }{L^\infty(K)}\]}
\\(ii) \textit{Proof of \eqref{conv_g2}}. 
Let $K$ be a compact subset of $\bar \O_g$ so that $\G_2 \subset K \subset \O_{g+\l h}$ (for all $\l\in (0,\l_0),$ $\l_0>0$ small) 
and $\interior K$ is a $\CC^{1,\alpha}$ domain ($\alpha>0$ as in \eqref{w2p_bound})  with $\clos{ \interior K}=K$. Note that such a  construction is possible since  $g,h\in \FF$ and in light of \eqref{hx00}. 
 From \eqref{conv_yl0} we know that \[y_{g+\l h} \weakly y_g \text{ in }W_0^{1,\varrho}(D), \quad \forall\,\varrho \in [1,\infty).\]Thanks to \eqref{w2p_bound} combined with \cite[Lem.\,6.36]{gt}, 
  it also holds on a subsequence  
  %\[y_{g+\l h} \weakly \widetilde y \text{ in }W^{2,p}(\interior K)\embed \embed \CC^{1,\alpha}(K),\]where $\alpha >0.$ From here we conclude 
  % $\widetilde y=y$ a.e.\,in $\interior K$  and since $\interior K$ is a non-empty domain and $y$ and $\widetilde y$ are both continuous on $K$, we have 
    \[y_{g+\l h} \to \widetilde y \text{ in } \CC^{1,\widetilde \alpha}(K),\]where $\alpha >\widetilde \alpha >0.$ From here we conclude 
  $\widetilde y=y_g$ in $K$
%Thus, by $W^{2,p}(\interior K) \embed \embed \CC^{1,\alpha}(K),$  we have 
% \[  y_{g+\l h} \to    y \text{ in }\CC^{1,\alpha}(K),\] 
 and, since $\G_2 \subset K$, the convergence \eqref{conv_g2} follows. The proof is now complete.\end{proof}

\subsection{Directional differentiability of $\SS$}\label{sec:dd}
Now we turn our attention to  our main goal in this section, which is to investigate the differentiability of $\SS$ w.r.t.\,$g\in \FF$ in a direction $h\in \FF$ satisfying \eqref{hx00}.

\subsection*{The vector field W and the candidate for the directional derivative}
 We start by introducing a vector field that will play an essential role in this paper. This associates to a point $x$ on $\partial \O_g$ its   velocity vector 
when performing
the variation of the domain. As we will see in section \ref{sec:sm}, this corresponds to the initial velocity vector field in the context of the speed method.

\begin{definition}\label{def:w}On the zero level set of $g$, we define the  vector field $W:\{g=0\}\to \R^2$ as follows 
 \begin{equation}W(x):=w((z^{-1}(x))\quad \forall\, x \in \{g=0\}^\circ,\quad {\{g=0\}^\circ}\subset \{g=0\}\end{equation}where $w\in \CC^1([0,T]; \R^2)$ is the unique solution of \eqref{w}.\end{definition}
 
We note the crucial  identity 
\begin{equation}\label{eq:wa}
W(z(t))=w(t), \quad t \in [0,T_{\{g=0\}^\circ}],\quad \forall\,{\{g=0\}^\circ}\subset \{g=0\}.\end{equation}
where $T_{\{g=0\}^\circ}>0$ is the period of ${\{g=0\}^\circ}$. This will be used at length throughout the proofs of Lemmas \ref{lem:h0} and \ref{lem:h1} below, when we look at the behaviour of the difference quotients $\frac{ y_{g+\l h}-y}{\l}$ on the boundary of the domain of interest.

In the following lemma we gather some important properties of  the vector field $W$.
\begin{lemma}\label{lem:W}Suppose that $g, h \in \FF$ and \eqref{hx00} is true. 
Then, the mapping $W$ satisfies 
\begin{equation}\label{eq:W} \nabla g(x)W(x)+h(x) =0 \quad \forall\, x \in \{g=0\}.\end{equation}
{Moreover, it is Lipschitzian in the following sense: there exists $L_W>0$ and $\delta_W>0$ so that for each $\{g=0\}^\circ \subset \{g=0\}$ and for all $x_1,x_2 \in \{g=0\}^\circ$ with $|x_1-x_2|\leq \delta_W$ we have 
\begin{align*}|W(x_1)-W(x_2)|\leq L_W|x_1-x_2|.
\end{align*}}
\end{lemma}
\begin{proof}(i) We prove the first assertion on each component of $\{g=0\}$. Define \begin{equation} w({ \l},\cdot)=\lim_{\hat \l \to   \l} \frac{z_{ \hat \l}-z_{  \l}}{\hat \l-  \l}\in \CC^1([0,T]; \R^2)\quad \forall\,\l \in (0,\l_0], \end{equation}where $\l_0>0$ is small and $z_{\hat \l},z_{\l}$ are the solutions of the Hamiltonian system \eqref{ham_e} associated to the respective component $\{g=0\}^\circ$.  The above limit  is computed in a similar way to $w$, see \eqref{wa}, cf.\,also the proof of \cite[Prop.\,3.2]{oc_t}. From Lemma \ref{lem:glh} we know that
\[(g+\l h)(z_{\l}(t))=0, \quad t \in (0,\infty), \quad \forall\,\l \in (0,\l_0]. \]
 Deriving with respect to $\l$, where we use Proposition \ref{prop:diff}, cf.\,also \eqref{wl0}, then yields
\[\nabla (g+\l h)(z_{\l}(t))w(\l,t)+h(z_{\l}(t)) =0, \quad t \in (0,\infty), \quad \forall\,\l \in (0,\l_0]. \]Setting $\l:=0$ gives in turn
that $w$ satisfies  \[\nabla g(z(t))w(t)+h(z(t)) =0, \quad t \in (0,\infty). \]In light of \eqref{eq:wa}, the proof of \eqref{eq:W} is complete.

(ii) We only show the estimate for a fixed component $\{g=0\}^\circ$; since their number is finite (Proposition \ref{prop:ham}), the desired assertion follows by building the minimum over all $\delta_W$'s and the maximum over all $L_W$'s. Let $z$ be the solution of the Hamiltonian system \eqref{ham} associated to  $\{g=0\}^\circ$. 
%We show that $W$ satisfies the desired Lipschitz continuity with $L_W:=\min_{\partial \O_g} |\nabla g|$ and $\delta_W:=$
%Note that $L_W>0$, since $\nabla g$ is continuous and $\partial \O_g$ is compact, by the Weierstrass theorem.
We observe that, since $z_i \in \CC^2[0,T], i=1,2$, it holds for all $\eps>0$
\[|z_i(t_1)-z_i(t_2)-z_i'(t_2)(t_1-t_2)|\leq {c}{(t_1-t_2)^2}\leq \eps |t_1-t_2|, \quad \forall |t_1-t_2|\leq \eps/c, \]where $c:=\frac{\|z_i''\|_{\CC[0,T]}}{2}>0$ depends only on $g$. Then,
\[(|z_i'(t_2)|-\eps)|t_1-t_2| \leq |z_i(t_1)-z_i(t_2)|, \quad \forall |t_1-t_2|\leq \eps/c.\]
Adding the above estimates for $i=1,2$ yields
\begin{equation}\label{1}(|z_1'(t_2)|+|z_2'(t_2)|-2\eps)|t_1-t_2| \leq \sum_{i=1}^2 |z_i(t_1)-z_i(t_2)|, \quad \forall |t_1-t_2|\leq \eps/c.\end{equation}
Further, in light of \eqref{ham}, it holds
\begin{equation}\label{2} |z_1'(t_2)|+|z_2'(t_2)|\geq \sqrt{\sum_{i=1}^2 |z_i'(t_2)|^2}=|\nabla g(z(t_2))|\geq \min_{\partial \O_g} |\nabla g| >0.\end{equation}
With $\eps:=\frac{1}{4}\min_{\partial \O_g} |\nabla g|>0$ (this is non-negative by the Weierstrass theorem, since $g \in \FF$ and $\partial \O_g$ is compact), we have
\[(\frac{1}{2}\min_{\partial \O_g} |\nabla g|)|t_1-t_2| \leq \sqrt{2}|z(t_1)-z(t_2)|, \quad \forall |t_1-t_2|\leq \eps/c.\]

Further we notice that $z^{-1}:\partial \O_g \to [0,T)$ is uniformly continuous, since it is continuous (as its inverse is continuous) on a compact set.
Let $x_1,x_2 \in \partial \O_g$ be arbitrary but fixed and define 
\[\delta_W:=\delta (\eps/c, z^{-1})>0\] so that $|z^{-1}(x_1)-z^{-1}(x_2)|\leq \eps/c$ for all $|x_1-x_2|\leq \delta_W$.
% with
%then $|z^{-1}(x_1)-z^{-1}(x_2)|\leq \eps/c$.
%\[\delta_W:=\frac{\eps}{cL_z},\]where $L_z>0$ is the Lipschitz constant of $z$; note that this depends only on $g$.
Abbreviate $t_1=z^{-1}(x_1)$ and $t_2=z^{-1}(x_2)$. Then, using \eqref{eq:wa} and the regularity of $w$, we get
\begin{equation}\label{ww}\begin{aligned}
|W(x_1)-W(x_2)|\leq |w(t_1)-w(t_2)|\leq L_w |t_1-t_2|\leq  \frac{2\sqrt 2}{\min_{\partial \O_g} |\nabla g|}|x_1-x_2|
\end{aligned}\end{equation}since $|z^{-1}(x_1)-z^{-1}(x_2)|\leq \eps/c$.
As $x_1,x_2 \in \partial \O_g$ were arbitrary with $|x_1-x_2|\leq \delta_W$, the desired Lipschitz continuity follows from \eqref{ww}, by setting $L_W:=\frac{2\sqrt 2}{\min_{\partial \O_g} |\nabla g|}$.
\end{proof}
%%Now, let $x\in  \{g\in (m,\g_0]$. Setting $\alpha:=g(x)$ and $t:=z^{-1}(x)$, where $x\in \partial \O_g$, then gives the desired  assertion. 

Now we are in the position to  introduce the candidate for the directional derivative of $\SS$ at $g$ in direction $h$. This is 
the unique solution $q \in W^{1,p}(\O_g)$ of 
\begin{equation}\label{q}
-\laplace  q +  \beta'(y_g;q) =0 \quad \text{in }  \O_{g}, \end{equation}
\begin{equation}\label{q0}
q+ \nabla y_gW =0 \quad \text{on } \partial \O_{g},\end{equation}
where $W:\partial \O_g \to \R^2$ is given by Definition \ref{def:w}. The fact that \eqref{q}-\eqref{q0} is indeed uniquely solvable is shown in Proposition \ref{prop:q} below. Before we proceed with its proof let us state an essential remark.
\begin{remark}\label{rem:sd}
The unique solution $q \in W^{1,p}(\O_g)$ of \eqref{q}-\eqref{q0} is the shape derivative of $y_{\O_g}$ in direction $W$ \cite[Def.\,2.85]{sz}. Indeed, when we compare with the results stated in \cite[Sec.\,3.1]{sz} (where, with the notations used there, we have $h(\O):=f-\beta(y_{\O})$ and $z(\G):=0$), we see that  \eqref{q} corresponds to \cite[Eq.\,(3.8)]{sz} and \eqref{q0} is in fact \cite[Eq.\,(3.6)]{sz}. In the latter case, it is essential to notice that $\nabla y_g W=\nabla y_g n_{\partial \O_g} (W n_{\partial \O_g})$ on $\partial \O_g$, where $n_{\partial \O_g}:\partial \O_g \to \R^2$ is the outer normal vector field of $\partial \O_g$; this can be proven by differentiating $y_g(z(t))=0,\ t\in (0,\infty)$ w.r.t.\,variable $t$, which yields that, on $\partial \O_g$, we have $\nabla y_g =c\,n_{\partial \O_g}$, where $c$ is a real-valued function defined on $\partial \O_g$. See also e.g.\,\cite[Eq.\,(5.28)-(5.29)]{hp} for a comparison in the linear case.
\end{remark}
\begin{proposition}\label{prop:q}Suppose that $g, h \in \FF$ and \eqref{hx00} is true. 
The Dirichlet problem with non-homogeneous boundary conditions \eqref{q}-\eqref{q0} admits a unique solution $q \in W^{1,p}(\O_g)$.
\end{proposition}
\begin{proof}
%\ko{that $W$ is locally lip continuous, even $C^1$, because $W(x):=w((z^{-1}(x))$ and $z^{-1}$ is $\CC^1??$inverse function thm(at least locally); or it follows fro  $w_\l$ lip?}. 
%\begin{align*}|W(z(t_1))-W(z(t_2))|&=|w(t_1)-w(t_2)|=|\int_{t_1}^{t_2}...|
%\\&\leq c|t_1-t_2|\leq ??|z_1'(\xi)||t_1-t_2|+|z_2'(\widehat \xi)||t_1-t_2|
%\\&=|z_1(t_1)-z_1(t_2)|+|z_2(t_1)-z_2(t_2)|
%\\&\leq |z(t_1)-z(t_2)|\end{align*}
Since $\beta'(y_g;\cdot):\R \to \R$ is monotone increasing (as a result of Assumption \ref{assu:stand}.\ref{it:stand2}), we can apply for instance \cite[Thm.\,4.8]{troe} to obtain  a first regularity result.
% $q \in H^1(\O_g) \cap \CC(\bar \O_g)$ and the estimate
%\begin{equation}\label{sup_q}\|q\|_{H^1(\O_g) \cap \CC(\bar \O_g)}\leq c,\end{equation}where $c>0$ depends only on $g,h \in \FF$ and the given data. 
This allows us to make use of the estimate in the standing Assumption \ref{assu:stand}.\ref{it:stand2} to deduce
$\beta'(y_g;q) \in {L^\infty(\O_g)} \embed W^{-1,p}(\O_g)$. 
%\[\|\beta'(y_g;q)\|_{L^\infty(\O_g)}\leq L_{M+1}\|q\|_{L^\infty(\O_g)}\leq C,\]where $C>0$ depends only on $g,h \in \FF$ and the given data. 
Thanks to {Lemma \ref{lem:W}.(ii), we conclude that $W\in W^{1,\infty}(\partial \O_g)$ (see e.g.\,\cite[Sec.\,5.8.2]{evans}), whence $ \nabla y_gW \in W^{1-1/p,p}(\partial \O_g)$ follows (recall that $y_g \in W^{2,p}(\O_g)$). Standard arguments from the $W^{1,p}(\O_g)$ regularity theory 
%simader thm 4,6
 then lead to the desired assertion.}
\end{proof}

\begin{remark}\label{rem:q}
As in the case of $y_g$, see Remark \ref{rem:S}, $q$  exists only as an element of $W^{1,p}(\O_g).$ Whenever we refer to the values of $q$ outside $\O_g$,  we think of its extension by zero, which yields $q \in L^\infty(D)$.
\end{remark}

\subsection*{Analysis of the difference quotient}
We continue our investigations by taking a closer look at  the term
\[ \frac{\SS(g+\l h)-\SS(g)}{\l}-q, \quad \l>0 \text{ small}.\]
 Given $g, h \in \FF$, we use for simplicity  in what follows the abbreviations:
 $\omega_\l:=\O_g \cap \O_{g+\l h}$, $q_\l:=\frac{ y_{g+\l h}-y}{\l}$ and $m_\l:= \|q_\l-q\|_{\CC(\partial \omega_\l)}$, where $\l>0$ is small enough (so that the unicity of the approximating closed curves is guaranteed, cf.\,Proposition \ref{lem:unique}).
\begin{lemma}\label{lem:lr}Suppose that $g, h \in \FF$ and \eqref{hx00} is true.  Then \[\||q_\l-q|-m_\l\|_{L^r(\widetilde \o_\l)}  \to 0, \quad \forall\,r\in [1,\infty),\]where $\widetilde \o_\l:=\{x \in \omega_\l:|q_{\l}-q|>m_\l\}$.\end{lemma}\begin{proof} Note that   $$q_\l \in W^{2,p}(\omega_\l),$$ by Definition \ref{S}. In light of \eqref{eq} (associated with $g$ and $g+\l h$, respectively) and \eqref{q}, it holds  
\begin{equation}\label{ol}-\laplace (q_\l-q) +[\frac{[\beta(y_{g+\l h}) -\beta(y_g)]}{\l}-\beta'(y_g;q)]=0 \quad \text{in }\omega_\l. \end{equation}
Observe that $(q_\l-q-m_\l)^+ \in H_0^1( \o_\l)$. This allows us to test the above equation with $(q_\l-q-m_\l)^+$. We obtain
\begin{equation}
\int_{ \o_\l} [ \nabla [q_\l-q-m_\l]^+]^2\,dx+\int_{ \o_\l} \underbrace{[\frac{[\beta(y_{g+\l h } )  -\beta(y_g)]}{\l}-\beta'(y_g;q)] }_{=: A_\l}  [q_\l-q-m_\l]^+\,dx =0.\end{equation}Note that 
\[A_\l=\frac{[\beta(y_{g+\l h} )  -\beta(y_g+\l q)]}{\l}+\frac{[\beta(y_g+\l q)  -\beta(y_g)]}{\l}-\beta'(y_g;q).\]
Due to the monotonicity of $\beta$ (Assumption \ref{assu:stand}.\ref{it:stand2}) and since $q_\l -q >m_\l (>0)$ implies $y_{g+\l  h }>y_g+\l q$, we can conclude that the first term in $A_\l$ is non-negative on $\{q_\l -q >m_\l \}$. 
This leads to 
\begin{equation}\label{q1''}
\int_{ \o_\l} [ \nabla [q_\l-q-m_\l]^+]^2\,dx+\int_{ \o_\l} \Big[\frac{[\beta(y_g+\l q) -\beta(y_g)]}{\l}-\beta'(y_g;q)\Big]  [q_\l-q-m_\l]^+\,dx \leq 0.\end{equation}
Since $[q_\l-q-m_\l]^+$ can be extended by zero outside $\o_\l$ (and it can thus be seen as an element in $H_0^1(D)$), we can apply Poincar\'e's inequality with a constant $c>0$ independent of $\l$ and obtain
\begin{equation}\|[q_\l-q-m_\l]^+\|_{L^2( \o_\l)}^2 \leq c\,\Big\|\frac{[\beta(y_g+\l q) -\beta(y_g)]}{\l}-\beta'(y_g;q) \Big\|_{L^{2}( \o_\l)} \|[q_\l-q-m_\l]^+\|_{L^{2}( \o_\l)}.\end{equation}
This yields \begin{equation}\label{q1'}\|[q_\l-q-m_\l]^+\|_{L^2( \o_\l)}\leq c\,\Big\|\frac{[\beta(y_g+\l q) -\beta(y_g)]}{\l}-\beta'(y_g;q) \Big\|_{L^{2}( \o_\l)} .\end{equation}
Adding $\|[q_\l-q-m_\l]^+\|_{L^2( \o_\l)}^2 $ on both sides in \eqref{q1''} then leads to 
\begin{equation}\begin{aligned}
&\|[q_\l-q-m_\l]^+\|_{H^1( \o_\l)}^2 \\& \leq \,\Big\|\frac{[\beta(y_g+\l q) -\beta(y_g)]}{\l}-\beta'(y_g;q) \Big\|_{L^{r'}( \o_\l)} \|[q_\l-q-m_\l]^+\|_{L^{r}( \o_\l)}
\\&\quad +\tilde c \|[q_\l-q-m_\l]^+\|_{L^r( \o_\l)}\|[q_\l-q-m_\l]^+\|_{L^2( \o_\l)},\end{aligned}  \end{equation}
where $r \in [2,\infty)$ is arbitrary and $\tilde c>0$ is the embedding constant of $L^r( \o_\l)\embed L^2( \o_\l)$; this is independent of $\l$, since $\mu( \o_\l)$ is bounded by a constant independent of $\l$. 
By using the same extension argument as above and the embedding $H^1(D) \embed L^r(D),$ as well as \eqref{q1'}, we arrive at 
\begin{equation}\label{q1}\|q_\l-q-m_\l\|_{L^r( \o_\l\cap\{q_\l -q>m_\l\})}\leq C \Big\|\frac{[\beta(y_g+\l q) -\beta(y_g)]}{\l}-\beta'(y_g;q) \Big\|_{L^{2}( \o_\l)} ,\end{equation}where $C:=1+\tilde c \, c>0$ is independent of $\l$.

In a similar manner, we obtain
\begin{equation}\label{q2}\|q-q_\l-m_\l\|_{L^r( \o_\l\cap\{q-q_\l >m_\l\})}\leq \widetilde C \Big\|\frac{[\beta(y_g+\l q) -\beta(y_g)]}{\l}-\beta'(y_g;q) \Big\|_{L^{2}( \o_\l)} ,\end{equation}where $\widetilde C>0$ is independent of $\l$, and we only sketch the arguments.

Testing \eqref{ol} with $(q-q_\l- m_\l)^+ \in H_0^1( \o_\l)$ yields 
\begin{equation}
\int_{ \o_\l} \nabla [q_\l-q] \nabla [q-q_\l-m_\l]^+\,dx+\int_{ \o_\l} \Big[\frac{[\beta(y_{g+\l h}) -\beta(y_g)]}{\l}-\beta'(y_g;q)\Big] [q-q_\l-m_\l]^+\,dx =0,\end{equation}
which accounts to
\begin{equation}-\int_{ \o_\l \cap \{q-q_\l>m_\l\}} (\nabla [q-q_\l-m_\l]^+)^2\,dx+\int_{ \o_\l \cap \{q-q_\l> m_\l\}} A_\l \underbrace{(q-q_\l- m_\l)}_{>0}\,dx 
=0.\end{equation}
Arguing as above we get that the first term in $A_\l$ is  negative on $\{q-q_\l  >m_\l \}$ (since $y_{g+\l h} <y_g+\l q$ there). 
Therefore \begin{equation}\int_{ \o_\l } (\nabla [q-q_\l-m_\l]^+)^2\,dx\leq \int_{ \o_\l} \Big[\frac{\beta(y_g+{\l }q )  -\beta(y_g)}{\l}-\beta'(y_g;q)\Big] {[q-q_\l- m_\l]^+}\,dx,\end{equation} 
%\begin{equation}\|[q-q_\l -m_\l]^+\|_{L^r(\O_{g+\l h})}  \leq \|\Big[\frac{[\beta(y+\l q) -\beta(y)]}{\l}-\beta'(y;q)\Big] \|_{L^r(\O_{g+\l h})} \to 0\end{equation}
%and thus
whence \eqref{q2} follows (by the exact same arguments as above). Combining \eqref{q1}, \eqref{q2} and \eqref{eq:dir} finally gives the desired assertion.
\end{proof}

\begin{proposition}\label{prop:lr}Suppose that $g, h \in \FF$ and \eqref{hx00} is true. If $m_\l \to 0,$ then \begin{equation}\label{lr}\|q_\l -q \|_{L^{r}(\o_\l)}\to 0 \quad \text{as }\l \searrow 0,\quad \forall\,r \in [1,\infty),\end{equation}and, for each open subset $\o \subset \subset \O_g$, we have 
{ \begin{equation}\label{h1}\|q_\l -q \|_{H^2(\o )}\to 0 \quad \text{as }\l \searrow 0.\end{equation}}\end{proposition}\begin{proof}Let $r \in [1,\infty)$ be fixed. 
In view of Lemma \ref{lem:lr} and since $m_\l \to 0,$ by assumption, we have
\begin{align*}\|q_\l-q \|_{L^r( \o_\l)} &\leq \|q_\l-q \|_{L^r(\widetilde \o_\l)}  +\|q_\l-q \|_{L^r( \{|q_{\l}-q|<m_\l \} )}
\\&\leq \||q_\l-q|-m_\l \|_{L^r(\widetilde \o_\l)} +\mu(\widetilde \o_\l)^{1/r} m_\l +\mu(\o_\l)^{1/r}m_\l
\\&\qquad \to 0 \quad \text{as }\l \searrow 0,\end{align*}where we recall $\widetilde \o_\l=\{x \in \omega_\l:|q_{\l}-q|>m_\l\} $, and use that $\o_{\l} \subset D$. This proves the first convergence.

We next show the second desired assertion.  Recall that $\o \subset \subset \O_g$ implies $\o \subset \subset \O_{g+\l h}$ for $\l>0$ small (cf.\,proof of Lemma \ref{lem:l}), and thus $\o \subset \subset \o_\l$. Due to  {\cite[Sec.6.3.1,\ Thm.\,1]{evans}} applied for \eqref{ol}, we get that $q_\l -q \in H^2(\o)$ and 
\begin{align*}\|q_\l -q\|_{H^2(\o)}\leq c \Big(\Big\|\frac{\beta(y_{g+\l h} )  -\beta(y_g)}{\l}-\beta'(y_g;q)\Big\|_{L^2(\o_{\l})}+\|q_\l -q\|_{L^2(\o_\l)}\Big),
\end{align*}
where  {$c>0$ depends only on $\frac{1}{\dist(\bar \o,\partial \o_\l)}$ and the given data, see \cite[Ch.\,3,\ Thm.\,4.2]{chen_wu}}; to see that $\frac{1}{\dist(\bar \o,\partial \o_\l)}$ is bounded from above by a constant independent of $\l$, define $\rho:=\dist (\bar \o,\partial \O)>0$ and use that $\partial \o_\l \subset \{x\in D:\dist (x,\partial \O)\leq \rho/2\}$ for $\l $ small enough (which is due to \eqref{h_d_c}). Then $ {\dist(\bar \o,\partial \o_\l)}\geq \rho/2$ and thus,  $c>0$ is independent of $\l$.

 Since $\beta$ satisfies \eqref{eq:flip}, the above estimate leads  to
\begin{align*}&\|q_\l -q\|_{H^2(\o)}\\&\leq c \Big(\Big\|\frac{\beta(y_{g+\l h} )  -\beta(y_g+\l q)}{\l}\Big\|_{L^2(\o_\l)}+\Big\|\frac{\beta(y_g+\l q)  -\beta(y_g)}{\l}-\beta'(y_g;q)\Big\|_{L^2(\o_\l)}
\\&\quad +\|q_\l -q\|_{L^2(\o_\l)}\Big)
\\&\leq c\Big(L_M\,\|q_\l -q\|_{L^2(\o_\l)}+\Big\|\frac{\beta(y_g+\l q)  -\beta(y_g)}{\l}-\beta'(y_g;q)\Big\|_{L^2(\o_\l)}\Big)
\\&\quad +c\|q_\l -q\|_{L^2(\o_\l)},
\end{align*} where $M:=\sup \{\|y_{g+\l h}\|_{L^\infty(\o_\l)}, \|y_g+\l q\|_{L^\infty(\o_\l)}\}$;  the fact that $\|y_{g+\l h}\|_{L^\infty(\o_\l)}$ is uniformly bounded is owed to \eqref{w2p_bound} (note that $\|y_g+\l q\|_{L^\infty(\o_\l)}$ is uniformly bounded as well since $\o_\l \subset \O_g$ and $y_g$ and $q$ are both continuous on $\O_g$).

Now, \eqref{h1} follows from \eqref{lr} and  \eqref{eq:dir}. This completes the proof.
\end{proof}

Given the result in Proposition \ref{prop:lr}, our next goal is to prove that $m_\l \to 0,$ by looking at the convergence behaviour of the difference quotient $\frac{ y_{g+\l h}-y}{\l}-q$ on the boundary of $ \o_\l=  \O_g \cap  \O_{g+\l h}$. To this end, we observe that \begin{equation}\label{b_ol}\partial \o_\l=\G_2 \cup (\partial \O_{g+\l h} \cap \O_g),\end{equation}where $\G_2$ is given by \eqref{g2}, that is, $\partial \o_\l$ consists of a finite union of (relatively closed) curves that  remain fixed with varying $\l$,  while $(\partial \O_{g+\l h} \cap \O_g)$ varies with changing $\l$. 
  \begin{figure}[ht]\label{fig:ol}
\begin{center}
\includegraphics[width=0.5\textwidth]{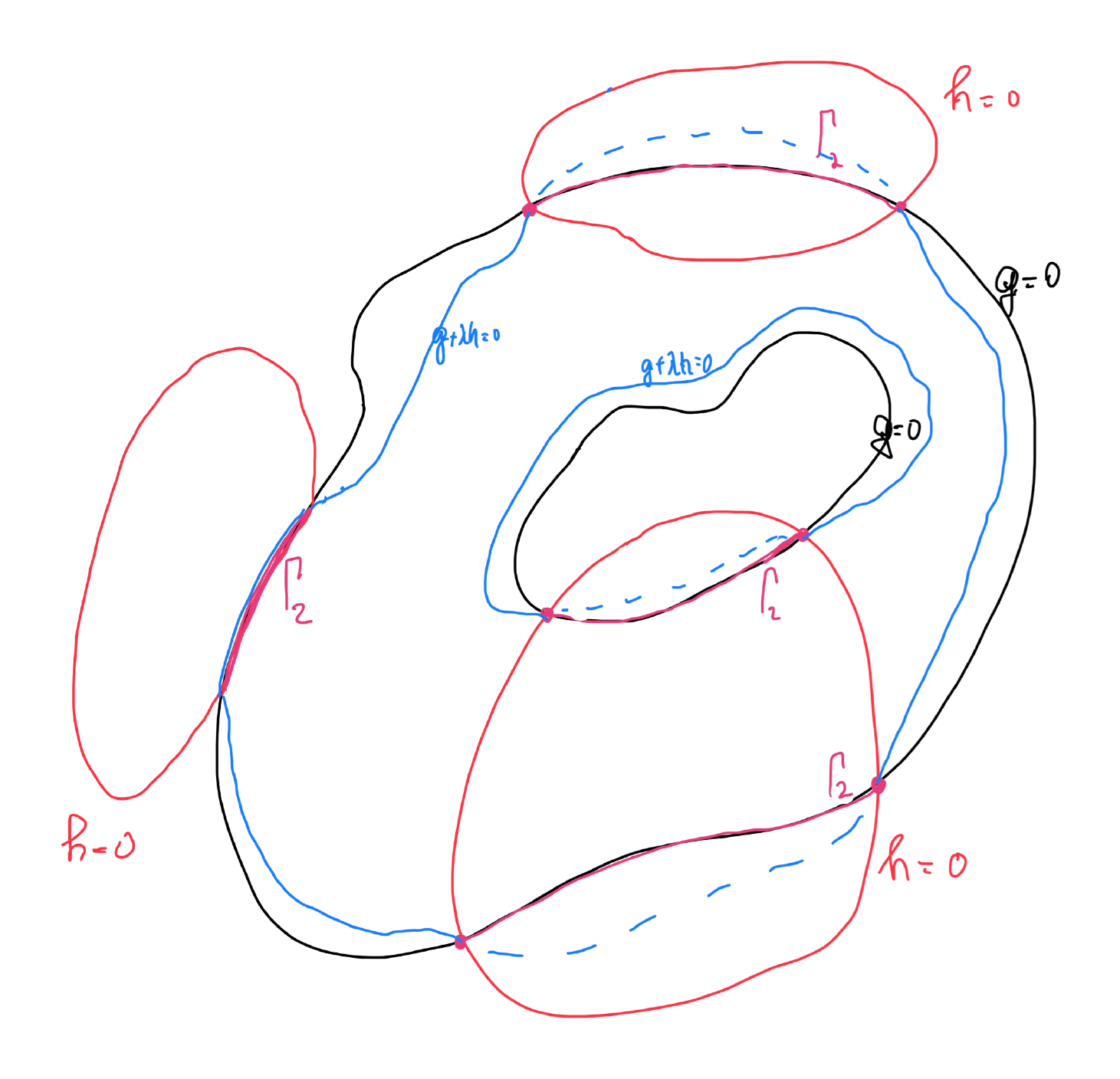}
\end{center}
\caption{The blue curves describe  $\partial \O_{g+\l h}$, while the black curves represent  $\partial \O_g$. The doted blue lines are those parts of $\partial \O_{g+\l h}$ that are outside $\bar \O_g$}
\end{figure}  
\begin{lemma}[Behaviour of the difference quotients on $\G_2$]\label{lem:h0}
Suppose that $g, h \in \FF$ and \eqref{hx00} is true. Then,
 \[\|q_{\l}-q\|_{\CC(\G_2)} \to 0\quad \text{as }\l \searrow 0.\]
  \end{lemma}
 \begin{proof}   Throughout this proof $\l>0$ is small enough, so that the unicity of the approximating curve is guaranteed, see Proposition \ref{lem:unique}. Since $h\in \FF$, by assumption, we have that $\{h=0\}$ is a finite union of disjoint closed $\CC^2$ curves (see Proposition \ref{prop:ham}). Thus, $\G_2$ is a finite union of {relatively closed} curves and it suffices to prove the assertion for one of this curves, which we call $\G_2$ as well (with a little abuse of notation). If $\G_2$ is such that $\G_2 \subset \partial \O_g \cap \partial \O_{g+\l h}$, then $y_{g+\l h}=0=y_g$ on $\G_2$ and $W=0$ on  $\G_2$ as well (Definition \ref{def:w} and Proposition \ref{prop:diff}), whence $\frac{y_{g+\l h}-y_g}{\l}=0=q$ on $\G_2$ follows.  
 Otherwise, take $x_0$ to be the initial point of $\G_2$, i.e., $h(x_0)=g(x_0)=0$. Consider the Hamiltonian systems \eqref{ham} and \eqref{ham_e} with this initial condition and associated solutions $z$ and $z_\l$, respectively; then $\G_2$ is completely described by $z:I\to \{g=0\}$, where $I=[0,\iota]$, for some $\iota>0$ and $z(\iota)\in \{h=0\} \cap \{g=0\}$ (this does not exclude the case $z(\iota)=x_0$). Thus, to prove the desired assertion, it suffices to show that
 \[(q_\l-q)(z(\cdot)) \to 0 \text{ uniformly on } I, \quad \text{as }\l \searrow 0.\]

 To obtain a first convergence result, we make use of  \cite[Lem.\,5.5]{dw_gw}. We observe that $\O_{g+\l h}$ is uniformly locally quasiconvex with constants independent of $\l$ \cite[Def.\,5.1]{dw_gw}; this is due to Lemma \ref{lem:glh} and since $|\nabla(g+\l h)|$ is uniformly bounded w.r.t.\,$\l$, for $\l>0$ small. In addition, thanks to \eqref{w2p_bound}, we can choose the constant $\delta>0$ in \cite[Lem.\,5.5]{dw_gw} independent of $\l$ as well (see its proof). Hence, applying  \cite[Lem.\,5.5]{dw_gw} for $y_{g+\l h} \in \CC^1(\O_{g+\l h})$ (where we use 
 %$z_\l \to z$ uniformly on $[0,T]$, cf.\,
 Proposition \ref{lem:vec}.(i) and the fact that $z(t)\in \O_{g+\l h}$, for $t \in I$) yields : $\forall \eps>0 $ there exists $\l(\eps)>0$ so that
 \[\max_{t \in I}\underbrace{\Big|\frac{y_{g+\l h}(z(t))-{y_{g+\l h}(z_\l(t))}}{\l}+ \nabla y_{g+\l h}(z(t))w_\l(t)\Big|}_{=:a_\l(t)}\leq \eps \max_{t \in I}|w_\l(t)| \ \,  \forall \l \in (0,\l(\eps)).\]
 Thus, as $\|w_\l\|_{\CC[0,T]}$ is uniformly bounded (see \eqref{wa}), we deduce \begin{equation}\label{al}a_\l \to 0 \text{ uniformly on }I.\end{equation} Since $y_g(z(t))=y_{g+\l h}(z_\l(t))(=0),\ t \in \R$, we have 
 \begin{align*}&\max_{t \in I}\Big|\frac{{y_{g+\l h}(z(t))}-y_g(z(t))}{\l}-q(z(t))\Big|
 \\&=\max_{t \in I}\Big|\frac{{y_{g+\l h}(z(t))}-y_{g+\l h}(z_\l(t))}{\l}\pm \nabla y_{g+\l h}(z(t))w_\l(t)\pm \nabla y_{g}(z(t))w_\l(t)-q(z(t))\Big|
 \\&\quad \leq  \max_{t \in I}a_\l(t)+\|\nabla y_{g+\l h}(z(\cdot))-\nabla y_g(z(\cdot))\|_{\CC(I)}\|w_\l\|_{\CC[0,T]}\\&\qquad + \|\nabla y_g(z(\cdot))\|_{\CC[0,T]}\|w_\l- w\|_{\CC[0,T]}, 
\end{align*} where we also employed \eqref{q0} and \eqref{eq:wa}.
 Due to \eqref{al},  \eqref{wa} and \eqref{conv_g2} we now arrive at the desired convergence.
 \end{proof}

 \begin{lemma}[Behaviour of the derivative on the varying part of the boundary of $\omega_\l$]\label{lem:h1}Suppose that $g, h \in \FF$ and \eqref{hx00} is true. We have 
 %\[\Big\|\frac{y_{g+\l h}(z_\l(\cdot))-y(z_\l(\cdot))}{\l}+(\nabla y W)(z(\cdot))\Big\|_{\CC[0,T]}  \to 0\]
% \[\Big\|\frac{y_{g+\l h}(z_\l(\cdot))-y(z_\l(\cdot))}{\l}+(\nabla y W)(z_\l(\cdot))\Big\|_{\CC[0,T]}  \to 0\]
% In particular,
 \[\|q_{\l}-q\|_{\CC(\partial \O_{g+\l h} \cap \bar  \O_g)} \to 0\quad \text{as }\l \searrow 0.\]   \end{lemma}  \begin{proof} Again, fix $\l>0$ small so that the unicity of the approximating curve is guaranteed, see Proposition \ref{lem:unique}. By the same arguments as in the proof of Lemma \ref{lem:h0} we have that  $\partial \O_{g+\l h} \cap \bar \O_g$ is a finite union of {(relatively closed)} curves and it suffices to prove the assertion for one of this curves, which we call $\G_\l$ in what follows (with a little abuse of notation). The case $\G_\l \subset \partial \O_g \cap \partial \O_{g+\l h}$ can be tackled as in the proof of Lemma \ref{lem:h0}. If $\G_\l \subset \O_g$, take $x_0$ to be the initial point of $\G_\l$, i.e., $h(x_0)=g(x_0)=0$. Consider the Hamiltonian systems \eqref{ham} and \eqref{ham_e} with this initial condition and associated solutions $z$ and $z_\l$, respectively; then $\G_\l$ is completely described by $z_\l:I_\l \to \{g+\l h=0\}$, where $I_\l=[0,\iota_\l]$, for some $\iota_\l>0$ and $z_\l(\iota_\l)\in \{h=0\} \cap \{g=0\}$. 
 Our goal is to show that
 \[\max_{t \in I_\l}|q_\l-q|\to 0\quad \text{as }\l \searrow 0.\]Arguing in the exact same way as in the proof of Lemma \ref{lem:h0} we obtain
 \begin{equation}\label{bl}\max_{t \in I_\l}\underbrace{\Big| \frac{ {y_g(z(t))}-y_g(z_\l(t))}{\l}+ \nabla y_g (z(t))w_\l(t)\Big | }_{=:b_\l(t)}\to 0.\end{equation}
 %For each  $t \in I_\l$  we have $z_\l(t)\in \O_g$ and 
 Further, we have for each  $t \in I_\l$ 
\begin{align*}\frac{\overbrace{y_{g+\l h}(z_\l(t))}^{=0}-y_g(z_\l(t))}{\l}-q(z_\l(t))&=\frac{\overbrace{y_g(z(t))}^{=0}-y_g(z_\l(t))}{\l}
\\&\quad \pm \nabla y_g (z(t))w_\l(t)\pm (\nabla y_g W)(z(t))-q(z_\l(t)).
%\\&=-\nabla y_g(\widetilde z_{\l}(t))w_\l(t)\pm \nabla y_g(z(t))w_\l(t)+\nabla y_g (z(t))w(t)\\&\quad +q(z(t))-q(z_\l(t)),
\end{align*}
%where $\widetilde z_\l (t) \in \O_g$. \ro{mean value thm ok because $z_\l (t)$ is inside $\O$}Note that here we used again \eqref{eq:wa} as well as \eqref{q0}. 
Thanks to  \eqref{eq:wa} and \eqref{q0}, this results in 
\begin{align*}&\max_{t \in I_\l} \Big|\frac{{y_{g+\l h}(z_\l(t))}-y_g(z_\l(t))}{\l}-q(z_\l(t))\Big|
\\&\leq \max_{t \in I_\l} b_\l(t)+ \|\nabla y_g(z(\cdot))\|_{\CC[0,T]}\|w_\l- w\|_{\CC[0,T]}\\&\quad +\widehat c \max_{t \in I_\l} |z(t)- z_\l(t)|^{  \alpha},\end{align*}
where   we used the embedding $  W^{1,p}(\O_g)\embed \CC^{0,\alpha}(\bar \O_g)$, $\alpha>0$ and the fact that $q \in W^{1,p}(\O_g)$. In view of \eqref{bl} and Propositions \ref{lem:vec} and \ref{prop:diff}, the proof is now complete.
\end{proof}

Summarizing, from Lemmas \ref{lem:h0}, \ref{lem:h1} and \eqref{b_ol}, we have 
\begin{proposition}\label{prop:ml}Suppose that $g, h \in \FF$ and \eqref{hx00} is true. It holds $$ \|q_\l-q\|_{\CC(\partial \omega_\l)} \to 0\quad \text{as }\l \searrow 0.$$ 
%As a consequence,
%\[\|q_\l -q \|_{L^{r}(\O_{g+\l h})}\to 0 \quad \text{as }\l \searrow 0,\quad \forall\,r \in [1,\infty),\]
%\ko{\[\|q_\l -q \|_{W^{1,p}(K)}\to 0 \quad \text{as }\l \searrow 0,\quad \forall\,K\subset \subset \O_{g+\l h}.\]}
\end{proposition}

We are now in the position to state a first differentiability result for $\SS $. This will 
allow us to prove the differentiability of the objective in the presence of derivatives of the state as well as pointwise observation on a prescribed subdomain, see  {subsection \ref{sec:gobs}}.

  \begin{theorem}\label{thm:h1}
 Assume that $g, h \in \FF$ and \eqref{hx00} is true.  Then 
\[\lim_{\l \searrow 0}\Big\|\frac{\SS(g+\l h)-\SS(g)}{\l}-q\Big\|_{H^2(\o)}\to 0,\]where $q \in W^{1,p}(\O_g)$ is the unique solution of \eqref{q}-\eqref{q0} and $\o \subset \subset \O_g$ is a compact  subset. 
%Hence, $\SS:\FF \to H^2(\o)$ is directionally differentiable at $g$ in direction $h$ with 
%\[\SS'(g;h)=q.\]
\end{theorem}
\begin{proof}The assertion is due to  Propositions \ref{prop:lr} and \ref{prop:ml}. \end{proof}

To state  a differentiability result for $\SS$ that permits us to have distributed observation in the unknown domain  in the objective as well (see subsection \ref{sec:do}), we will also need to take a look at the convergence behaviour of  $q_\l -q$ outside $\o_\l$ (e.g.\, on $\O_{g+\l h} \setminus \O_g$ and on $\O_g \setminus \O_{g+\l h}$).  
  
  \begin{proposition}\label{prop:outside}
  Suppose that $g, h \in \FF$ and \eqref{hx00} is true. We have 
  \[\|q_\l-q\|_{L^r(\O_g \setminus \O_{g+\l h})}\to 0, \quad \forall\,r\in [1,\infty),\]
  \[\|q_\l-q\|_{L^r(\O_{g+\l h} \setminus \O_g)}\to 0\quad \forall\,r\in [1,\infty),\]where we recall $q_\l=(y_{g+\l h} -y_g)/\l$ and that $y_{g+\l h}$ is extended by zero outside $\O_{g+\l h}$, while $y_g$ and $q$ are extended by zero outside $\O_g$.
\end{proposition}

\begin{proof}Let $r \in [1,\infty)$ be arbitrary but fixed. Take $\l>0$ small enough and fix $\varrho \in (r,\infty)$. We first show that
 \begin{equation}\label{ql_bound}
\|q_\l\|_{L^\varrho(D)} \leq c, \end{equation}where $c>0$ is independent of $\l$.
 \\(i) If $x\in \O_{g+\l h} \setminus  \O_g$, then $q_\l(x)=\frac{y_{g+\l h}(x)}{\l}$. By  \eqref{w2p_bound}, one has  \[|y_{g+\l h}(x)-\underbrace{y_{g+\l h}(\widehat x)}_{=0}|\leq L |x-\widehat x|\leq L\l,\]where $\widehat x \in \partial \O_{g+\l h}$ and  $L>0$ is independent of $\l$. Hence, it holds $\|q_\l\|_{L^\infty(\O_{g+\l h} \setminus  \O_g)}\leq L.$
 \\(ii) If $x\in \O_g \setminus \O_{g+\l h}$, we argue as above by taking into account  that  $q_\l(x)=\frac{-y_g(x)}{\l}$ and $|y_g(x)-y_g(\widehat x)|\leq L \l,$ where $\widehat x \in \partial \O_g$; we obtain $\|q_\l\|_{L^\infty(\O_g \setminus  \O_{g+\l h})}\leq L.$
\\(iii) If $x\not \in \O_{g+\l h} \cup \O_g$, then $q_\l(x)=0.$ Due to Proposition \ref{prop:ml}, we also have that $\|q_\l\|_{L^\varrho(\o_\l)}$ is uniformly bounded. From all the above, we arrive at 
  \eqref{ql_bound}. Further, by \eqref{mu}, one has 
  \[\|q_\l-q\|_{L^r(\O_g \setminus \O_{g+\l h})}\leq \mu(\O_g \setminus \O_{g+\l h})^{r\varrho/\varrho-r}\|q_\l-q\|_{L^\varrho(D)}\to 0,\]
  \[\|q_\l-q\|_{L^r(\O_{g+\l h} \setminus \O_g)}\leq \mu(\O_{g+\l h} \setminus \O_g)^{r\varrho/\varrho-r}\|q_\l-q\|_{L^\varrho(D)}\to 0,\]
and the desired assertion is now proven. \end{proof}

The next theorem contains the second main differentiability result for $\SS$.

\begin{theorem}\label{thm:lr}
Assume that $g, h \in \FF$ and \eqref{hx00} is true.  Then 
\[\lim_{\l \searrow 0}\Big\|\frac{\SS(g+\l h)-\SS(g)}{\l}-q\Big\|_{L^r(\O_g)}\to 0,\quad \forall\,r \in [1,\infty),\]where $q \in W^{1,p}(\O_g)$ is the unique solution of \eqref{q}-\eqref{q0}. Moreover, 
\[\lim_{\l \searrow 0}\Big\|\frac{\SS(g+\l h)-\SS(g)}{\l}-q\Big\|_{L^r(D)}\to 0,\quad \forall\,r \in [1,\infty),\]when $q \in W^{1,p}(\O_g)$ is extended by zero outside $\O_g$. Hence, $\SS:\FF \to L^r(D)$ is directionally differentiable at $g$ in direction $h$ with 
\[\SS'(g;h)=q.\]
\end{theorem}
 \begin{proof}
 The result is due to Propositions \ref{prop:ml} and \ref{prop:outside}.
 \end{proof}

We end this section by recalling that $q$ is in fact the shape derivative of $y_{\O_g}$ in direction $W$ \cite[Def.\,2.85]{sz}, cf.\,Remark \ref{rem:sd}. Hence, in view of Theorem \ref{thm:lr}, we have obtained  
\begin{equation}
\lim_{\l \searrow 0}\frac{\SS(g+\l h)-\SS(g)}{\l}=\lim_{\l \searrow 0}\frac{y_{T^{\WW}_\l(\O_g)}-y_{\O_g}}{\l}=q \text{ in }L^r(D),\end{equation}
where $T^{\WW}_\l(\O_g)$ is the flow generated by a family of vector fields $\WW$ for which $\WW(0)=W.$ We stress that our perturbations happen only at the boundary level, and that, as a consequence, the whole  domain gets then automatically perturbed. An additional (rather formal) confirmation of this statement is given in section \ref{sec:sm} below, where we   present the connection  between the FVA and the speed method.
\section{Equivalence of the shape optimization problem with a  control problem with admissible set consisting of functions}\label{sec:so}
In this section we turn our attention to shape optimization problems governed by \eqref{eq}, such as 
\begin{equation}\tag{\ensuremath{P_\O}}\label{p_sh}
 \left.
 \begin{aligned}
  \min_{\O \in \OO, E \subset \subset \O} \quad &\JJ(\O, y_\O)\\     \text{s.t.} \quad & 
  \begin{aligned}[t]
    -\laplace y_{\O} + \beta(y_{\O})&=f \quad \text{{in} }\O,
   \\y_{\O}&=0  \quad \text{on } \partial \O,\end{aligned} \end{aligned}
 \quad \right\}
\end{equation}where $\JJ:{\mathcal{C}_2 }\times H_0^1(D) \to \R$ is a given objective functional and $\CC_2$ is the set of all subsets of $D$ of class $\CC^2$. The set of admissible shapes   is given by 
\[\OO:=\{\O \subset D: \O \text{ is a  non-empty } \CC^{2} \text{ domain},\ \partial \O \cap \partial D=\emptyset\},\]
while the observation set $E\subset \subset D$ is a non-empty subdomain of $D$.%The standing assumptions imposed on $\beta$ and $f$ are supposed to further hold and we recall that, for $\O\in \OO$, the state $y_\O \in W^{2,p}(\O) \cap H^1_0(\O)$, cf.\,Definition \ref{S}.

By arguing exactly as in the proof of  \cite[Prop.\,2.8]{p1}, one shows  that
\begin{equation}\label{o}\OO=\{\O_g:g \in \FF\},\end{equation} where we recall 
 \begin{equation}\label{f'}\FF=\{g \in \CC^{2}(\bar D):\ |\nabla g(x)|+|g(x)|>0 \ \forall\,x\in D,\ g(x)>0 \ \forall\,x\in\partial D, \ g^{-1}((-\infty,0] )\neq \emptyset\}. \end{equation}
 Note that the last condition in \eqref{f'} ensures that $\O_g \neq \emptyset$.
While the inclusion $"\subset"$ in \eqref{o} is due to arguments inspired by \cite[Thm.\,4.1]{dz}, the opposite inclusion follows by \cite[Thm.\,4.2]{dz}. The latter  gives that, for $g \in \FF$, the set $\O_g$ is of class $\CC^2$ (and its boundary does not intersect $\partial D$); but this is   not necessarily a domain. However, by \cite[Lem.\,2.9]{p1}, one can find another $\hat g \in \FF$ so that $\O_{\hat g} \in \OO$.

In the sequel we abbreviate for simplicity \[\OO_E:=\{\O\in \OO:E \subset \subset \O\}\] and 
\[\FF_E:=\{g \in \FF: g<0 \text{ in }\bar E\}.\] Note that, due to the above arguments, it holds
\begin{equation}\label{oe}\OO_E=\{\O_g:g \in \FF_E\}.\end{equation}

As in \cite{p1}, we want to rewrite  \eqref{p_sh}  as an optimal control problem where the admissible set consists of functions, namely \begin{equation}\label{p}\tag{\ensuremath{P}}
 \left.
 \begin{aligned}
  \min_{g \in \FF_E} \quad & \JJ(\O_g,y_g)\\     \text{s.t.} \quad & 
  \begin{aligned}[t]
    -\laplace y_{g} + \beta(y_{g})&=f \quad \text{in }\O_g,
   \\y_{g}&=0  \quad \text{on } \partial \O_g. \end{aligned} \end{aligned}
 \quad \right\}
\end{equation}To ensure that \eqref{p_sh} is equivalent to \eqref{p}, we need to ask that $\JJ$ satisfies for all $\O_1\in \OO,\O_2 \in \CC_2$ the estimate
\begin{equation}\label{j}
\JJ(\O_1,y_{\O_1}) \leq \JJ(\O_2,y_{\O_2}), \quad \text{if }\O_1 \subset \O_2.
\end{equation}

Then, the above mentioned equivalence is to be understood as follows.

\begin{proposition}\label{rem:equiv}
Assume that $\JJ$ fulfills \eqref{j}. Let  $\O^\star \in \OO_E$ be an optimal shape of \eqref{p_sh}. Then, each of the  functions  $g^\star \in \FF_E$ that satisfy $\O_{g^\star}=\O^\star$ is a global minimizer of \eqref{p}. Conversely, if $g^\star \in \FF_E$ minimizes \eqref{p}, then the component of $\O_{g^\star}$ that contains $\bar E$ is an optimal shape for \eqref{p_sh}.
\end{proposition}
\begin{proof}We follow the lines of the proof of \cite[Prop.\,2.10]{p1}. 
Let $ \O^\star \in \OO$ be an optimal shape for \eqref{p_sh}, i.e.,
\[\JJ(\O^\star , y_{ \O^\star})\leq \JJ(\O  , y_{ \O }) \quad \forall\,\O \in \OO_E.
\]
Now, let $g^\star \in \FF$ with $\O_{g^\star}=\O^\star$ be fixed (note that, according to \cite[Lem.\,2.9]{p1}, there are infinitely many mappings with this property).  
Then, 
$y_{\O^\star}=y_{ g^\star}$ and 
\[\JJ(\O_{g^\star} , y_{g^\star})\leq \JJ(\O  , y_{ \O }) \quad \forall\,\O \in \OO_E.\]
Let $g \in \FF_E$, be arbitrary and fixed; define $\hat g \in \FF$ so that $\O_{\hat g}\in \OO$ is the component of $\O_g$ that contains $\bar E$. Testing with $\O_{\hat g}$ in the above inequality yields 
\begin{align*}\JJ(\O_{g^\star} , y_{g^\star})\leq \JJ(\O_{\hat g}  , y_{ \hat g })
&\leq \JJ(\O_{  g}  , y_{  g }) ,\end{align*}
where in the second estimate we  used {\eqref{j}}. This proves the first statement.
\\
To show the converse assertion, assume that $g^\star \in \FF_E$ satisfies
\begin{equation}\label{inn}
\JJ(\O_{g^\star} , y_{g^\star})\leq  \JJ(\O_{  g}  , y_{  g }) \quad \forall\,g \in \FF_E.\end{equation}
We denote by $\o_{g^\star}$  the component of $\O_{g^\star}$ that contains $\bar E$. 
%This implies 
%\[y_{\o_{g^\star}}=y_{ g^\star} \quad \text{in }\o_{g^\star}.\]
Let $\O \in \OO_E$ be arbitrary but fixed. Again, we can define $\widetilde g \in \FF_E$ so that $\O_{\widetilde g}=\O.$ Then,
$y_{\O}=y_{\widetilde g}$. In view of \eqref{inn}, where we test with $\widetilde g$, we have
\begin{align*}
\JJ(\o_{g^\star} , y_{g^\star})
\leq \JJ(\O_{g^\star} , y_{g^\star})\leq \JJ(\O_{\widetilde g} , y_{\widetilde g})=\JJ(\O  , y_{\O}),
\end{align*}where in the first estimate we used again \eqref{j}. This proves the second assertion and completes the proof.
\end{proof}

With the result in Proposition \ref{rem:equiv} at hand, the reduced control problem we are interested in the sequel 
 reads
\begin{equation}\label{red_c_pb}\min_{g \in \FF_E} \quad j(g),\end{equation}where 
we abbreviate $$j(g):=\JJ(\O_g, y_g)$$ in all what follows. 

\subsection{Local optimality}\label{sec:lo}
In this subsection, we go a step further and establish a notion of local optimal shape for \eqref{p_sh} and its relation with local optimality for control problems such as \eqref{red_c_pb}.

\begin{definition}\label{def:loc0}We say that $\O^\star \in \OO_E$ is   locally optimal for \eqref{p_sh} if there exists $R>0$ so that
\[\JJ(\O^\star , y_{ \O^\star})\leq \JJ(\O  , y_{ \O }) \quad \forall\,\O \in \OO_E \text{ with }d_{\HH}(\partial \O^\star, \partial \O)\leq R .
\]\end{definition}
%\begin{assumption}[compatibility ]

%\end{assumption}

%\begin{definition}\label{def:loc}We say that $g^\star \in \FF_E $ is locally optimal for \eqref{p} if there exists $r>0$ so that
%\begin{align*}\JJ(\O_{g^\star} , y_{g^\star})\leq  \JJ(\O_{  g}  , y_{  g }) \quad &\forall\,g\in \FF_E \text{ with }\|g-g^\star\|_{\CC^1(\bar D)}\leq r \text{so that }\O_g \in \OO_{N_{g^\star}},
%\\&\text{ and }\{g=0\}\cap \{g^\star =0\}^\circ \neq \emptyset \quad \forall\,\{g^\star =0\}^\circ\subset \{g^\star=0\}.
%\end{align*}\end{definition}
%\begin{equation}\label{gx00}\{g=0\}\cap \{g^\star =0\}^\circ \neq \emptyset \quad \forall\,\{g^\star =0\}^\circ\subset \{g^\star=0\},\end{equation}and 
\begin{theorem}\label{prop:loc}
Let  $\O^\star \in \OO_E$ be a locally optimal shape of \eqref{p_sh}. Then, each of the  functions  $g^\star \in \FF_E$ that satisfy $\O_{g^\star}=\O^\star$ is a "local minimizer" of \eqref{p} in the following (restricted) sense: There exists 
$r>0$ so that
\begin{equation}
\begin{aligned}\label{loc_opt}j({g^\star})\leq  j({  g }) \quad &\forall\,g\in \FF_E \text{ with }\|g-g^\star\|_{\CC^1(\bar D)}\leq r \text{ so that }\O_g \in \OO,
\\&\{g=0\}\cap \{g^\star =0\}^\circ \neq \emptyset \quad \forall\,\{g^\star =0\}^\circ\subset \{g^\star=0\},
\\&\text{ and }\{g^\star=0\} \cap \{g  =0\}^\circ \neq \emptyset \quad \forall\,\{g  =0\}^\circ\subset \{g=0\}.\end{aligned}\end{equation}
\end{theorem} 
\begin{remark}
The concept of local optimality  for \eqref{red_c_pb}  in Theorem \ref{prop:loc}   may appear a bit restrictive. However,  the admissible set in \eqref{loc_opt} allows for test functions such as $g+\l h$, provided that $\l$ is  small enough  and $h \in \FF$ satisfies \eqref{hx00}. Since this is the only type of functional variation  we are interested in, Theorem \ref{prop:loc} serves its purpose and we can provide first-order necessary optimality conditions for all locally optimal shapes of \eqref{p_sh}, cf.\,Theorem \ref{prop:necc} below.
\end{remark}
\begin{proof}
Let $g^\star \in \FF_E$ with $\O_{g^\star}=\O^\star$ be fixed.  According to Proposition \ref{prop:ham}, $\{g^\star =0\}$ constitutes of a finite number of closed curves and we denote in the sequel by $T>0$ the largest period among these curves.

Define \begin{equation}\label{r}r:=\frac{R}{T\exp(2T\sup\{L_{\partial_1 g^\star},L_{\partial_2 g^\star}\})}>0,\end{equation}where $L_{\partial_i g^\star}>0,i=1,2$ denotes the Lipschitz constant of ${\partial_i g^\star}$ and $R>0$ is the radius of local optimality of $\O^\star$.
Let $g \in \FF_E$ with  
\begin{equation}\label{c1d}\|g-g^\star\|_{\CC^1(\bar D)}\leq r \end{equation} and $\O_g \in \OO$ so that 
%$\{g=0\}$ has the same number of components as $\{g^\star   =0\} $ and 
 \begin{equation}\label{gx00}\{g=0\} \cap \{g^\star =0\}^\circ \neq \emptyset \quad \forall\,\{g^\star =0\}^\circ\subset \{g^\star=0\},\end{equation}
 \begin{equation}\label{gx00'}\{g^\star=0\} \cap \{g  =0\}^\circ \neq \emptyset \quad \forall\,\{g  =0\}^\circ\subset \{g=0\}.\end{equation}  
  Our aim is to show that $d_{\HH}(\partial \O^\star, \partial \O_{g})\leq R$. 
 
% First and foremost we observe that 
% \begin{equation}\label{gx00'}\{g^\star=0\} \cap \{g  =0\}^\circ \neq \emptyset \quad \forall\,\{g  =0\}^\circ\subset \{g=0\},\end{equation}  which is a consequence of $\O_g \in \OO^\star$. In fact, each closed curve $\{g^\star =0\}^\circ$ has intersection points with exactly one component of $\{g  =0\}$ and viceversa.
 %be arbitrary and fixed. Define $\hat g \in \FF$ so that $\O_{\hat g}\in \OO$ is the component of $\O_g$ that contains $\bar E$. We want to show that $d_{\HH}(\partial \O^\star, \partial \O_{\hat g})\leq R $ so that we can test with $\O_{\hat g}$ in Definition \ref{def:loc0}. 

%we note that $g=\hat g$ in $\bar \O_{\hat g}$. recall that the number of components of both boundaries is finite. 
Take $\{g^\star=0\}^\circ$. In light of \eqref{gx00} there is a component of  $\{g=0\}$ that intersects $\{g^\star=0\}^\circ$, say $\{g=0\}^\circ$. Then,  we can choose
in the Hamiltonian system  \eqref{ham} associated with $\{g^\star=0\}^\circ$  the initial condition $ {x}_0\in \{g=0\}^\circ$, so that 
 \begin{equation}\label{gx0}g^\star( {x}_0) = g(x_0)=0.\end{equation}

Estimating as in the proof of \cite[Prop.\,3.10]{p3}   we obtain
 \[
 | z_{g^\star} (t)-z_{g}(t)| \leq   2\sup\{L_{\partial_1 g^\star},L_{\partial_2 g^\star}\}\int_0^t  | z_{g^\star} (s)-z_{g}(s)|\,ds +t\|g^\star -{g}\|_{\CC^1(\bar D)}\quad \forall\,t \in \R_{+},\]
and Gronwall's lemma gives in turn 
  \begin{equation}\label{gronw0}
 \begin{aligned}
 | z_{g^\star} (t)-z_{g}(t)| &\leq Ct\|g^\star -{g}\|_{\CC^1(\bar D)}\quad \forall\,t \in \R_{+},\end{aligned}
\end{equation}where \begin{equation}\label{c}C:=\exp(2T\sup\{L_{\partial_1 g^\star},L_{\partial_2 g^\star}\}).\end{equation}According to \cite[Prop.\,2.2.27]{hp}, it holds  
 $$d_{\HH}(\partial \O^\star, \partial \O_{g})=\|\dist_{\partial \O^\star}-\dist_{\partial\O_{g}}\|_{L^\infty(\partial \O^\star\cup \partial \O_{g})}.$$ 
Take $x\in \partial \O^\star\cup \partial \O_{g}$ arbitrary but fixed. If $x \in \partial \O^\star$, then $x$ belongs to a closed curve $\{g^\star=0\}^\circ$ and there is $t\in[0,T]$ so that $x=z_{g^\star}(t)$. Then
\[| \dist_{\partial\O_{g}}x|\leq 
|z_{g^\star}(t)-z_{g}(t)|\leq CT\|g^\star -{g}\|_{\CC^1(\bar D)},
\] where $z_{g}$ describes a closed curve $\{g =0\}^\circ$ intersecting $\{g^\star=0\}^\circ$ (that is, we can choose $z_{g^\star}(0)=z_{g}(0)=  x_0$ which allows us to make use of \eqref{gronw0}).
We estimate in the exact same manner for $x \in \partial \O_{g}$. We underline that this is possible due to \eqref{gx00'}. To summarize we obtained 
 $$d_{\HH}(\partial \O^\star, \partial \O_{g})\leq CT\|g^\star -{g}\|_{\CC^1(\bar D)}\leq \exp(2T\sup\{L_{\partial_1 g^\star},L_{\partial_2 g^\star}\})Tr =R,$$where we relied on \eqref{c}, \eqref{c1d} and \eqref{r}. 
 
 Finally we recall that $\O_{g^\star}$ is locally optimal for \eqref{p_sh}, cf.\,Definition \ref{def:loc0}, i.e.,\[\JJ(\O_{g^\star} , y_{g^\star})\leq \JJ(\O  , y_{ \O }) \quad \forall\,\O \in \OO_E \text{ with }d_{\HH}(\partial \O^\star, \partial \O)\leq R .
\] Since $\O_g \in \OO_E$ (as $\O_g \in \OO$ and $g \in \FF_E$)
we deduce \[\JJ(\O_{g^\star} , y_{g^\star})\leq \JJ(\O_{g}  , y_{ \O_{g}  }) \]which  completes the proof.
%Then, 
%$y_{\O^\star}=y_{ g^\star}$ and 
%\[\JJ(\O_{g^\star} , y_{g^\star})\leq \JJ(\O  , y_{ \O }) \quad \forall\,\O \in \OO_E \text{ with }d_{\HH}(\partial \O^\star, \partial \O)\leq R ,
%\]where we use \eqref{c} and \eqref{r}.
%Testing with $\O_{g}$ in the above inequality yields 
%\begin{align*}\JJ(\O_{g^\star} , y_{g^\star})\leq \JJ(\O_{\hat g}  , y_{ \hat g })
%&\leq \JJ(\O_{  g}  , y_{  g }) ,\end{align*}
%where in the second estimate we  used {\eqref{j}}. 
\end{proof}

We end this section with an essential observation concerning the opposite implication in Theorem \ref{prop:loc}.
\begin{remark}
The fundamental argument standing at the core of the proof of Theorem \ref{prop:loc} is that if two parametrizations are "close" to each other, then their associated shapes will be "close" too. The reverse statement is however not true.
If  two   shapes are identical even, this does not mean that the same is true for their parametrizations, on the contrary:
%since, given $g\in \FF$, there is an infinity of functions in $\FF$  generating the same $\O_g$  \cite[Lem.\,2.9]{p1}. In fact,
for each $ g \in \FF$, it holds $\O_{ g}=\O_{n\,g}, \ n \in \N_+,$ that is,  $\|g_n-g\|_{\CC^1(\bar D)} \to \infty$ as $n \to \infty$ for $g_n=n\,g$.  Hence, it is not clear if the local optimality of some $g^\star$ for the control problem \eqref{red_c_pb} implies that $\O_{g^\star}$ is  a  locally optimal shape for \eqref{p_sh}.

%\\
%On the other hand, if $\|g_n-g\|_{\ld} \to 0$, $g_n, g \in \FF_{\mathfrak{s}}$, we have
%\[\mu\{x \in D: g>0 \text{ and }g_n \leq 0\} \to 0 \quad \text{as }n \to \infty,\]
%\[\mu\{x \in D: g<0 \text{ and }g_n \geq 0\} \to 0 \quad \text{as }n \to \infty,\] see Lemma \ref{lem:app}. 
%In light of Lemma \ref{fs}, this accounts to
%\[\mu(\O_{g_n}\setminus \O_{ g})\to 0 \quad \text{as }n \to \infty,\]
%\[\mu( \O_{ g}\setminus \O_{ g_n})\to 0 \quad \text{as }n \to \infty.\]
%It is an open question how to define the notion of 'local optimal shape' so that it makes sense in connection with the $L^2$ local optimality of the associated parametrizations; this may involve the Hausdorff-Pompeiu distance or other notions of distances  between two sets. The above convergences may be a good starting point for further investigations on this particular topic.
%At this stage, it is however clear that our method covers all global optimal shapes (Proposition \ref{rem:equiv} and Corollary \ref{cor:os}). 
\end{remark}
\section{Necessary optimality conditions for locally optimal shapes}\label{sec:noc}
With the results from the previous section at hand, we can write a necessary optimality condition in primal form for each  locally optimal shape  of \eqref{p_sh}.
\begin{theorem}\label{prop:necc}
 Let  $\O^\star \in \OO_E$ be a locally optimal shape of \eqref{p_sh} in the sense of Definition \ref{def:loc0}. Then, each of the  functions  $g^\star \in \FF_E$ that satisfy $\O_{g^\star}=\O^\star$ fulfill
\[j'(g^\star; h)\geq 0 \quad \forall\, h \in \FF \text{ with }\{h=0\}\cap \{g^\star =0\}^\circ \neq \emptyset \quad \forall\,\{g^\star =0\}^\circ\subset \{g^\star=0\},\]provided that $j$ is directionally differentiable at $g^\star$ in direction $h$.
\end{theorem}\begin{proof}
Thanks to Theorem \ref{prop:loc}, we have  \begin{equation}\label{jj}\begin{aligned}
j({g^\star})\leq  j({  g })  \quad &\forall\,g\in \FF_E \text{ with }\|g-g^\star\|_{\CC^1(\bar D)}\leq r \text{ so that }\O_g \in \OO,
\\&\{g=0\}\cap \{g^\star =0\}^\circ \neq \emptyset \quad \forall\,\{g^\star =0\}^\circ\subset \{g^\star=0\},
\\&\text{ and }\{g^\star=0\} \cap \{g  =0\}^\circ \neq \emptyset \quad \forall\,\{g  =0\}^\circ\subset \{g=0\},\end{aligned}\end{equation}
with some $r>0$.
By using Weierstrass theorem, see for instance \cite[Lem.\,3.8]{p3}, one can show that, given $h \in \CC^2(\bar D)$, it holds 
\[g^\star \in \FF_E \Rightarrow g^\star+\l h=:\widehat g \in \FF_E \]
for $\l>0$ small, dependent only on $g^\star$ and $h$ and the given data. Moreover, $\|\widehat g-g^\star\|_{\CC^1(\bar D)}=\l\|h\|_{\CC^1(\bar D)}\leq r, $ by choosing $\l>0$ small enough. Now let $h \in \FF$ with \[\{h=0\}\cap \{g^\star =0\}^\circ \neq \emptyset \quad \forall\,\{g^\star =0\}^\circ\subset \{g^\star=0\}\]
be arbitrary but fixed and fix also  $\{g^\star =0\}^\circ\subset \{g^\star=0\}$. Then there is $x_0 \in \{g^\star =0\}^\circ$ so that $h(x_0)=0$, which means that $(g^\star+\l h)(x_0)=0$. That is, the second line in \eqref{jj} is satisfied by $\hat g=g^\star+\l h.$ If $\l>0$ is chosen so small that $\{\widehat g=0\}^\circ$ is the unique approximating curve of $\{g^\star=0\}^\circ$ (cf.\,Proposition \ref{lem:unique}), then each of the components of $\{\widehat g=0\}$ intersects $\{g^\star=0\}$. This implies that  the third  line in \eqref{jj} is satisfied by $\widehat g=g^\star+\l h$ as well. 

In view of all the above, $\widehat g$ is admissible for \eqref{jj} and we obtain
\[j(g^\star)\leq j(g^\star+\l h)\quad \forall \l \in (0,\l_0(h,g^\star)),\]
where $\l_0(h,g^\star)>0$ is small enough. Hence, dividing by $\l$, rearranging the terms  and using that $j$ is directionally differentiable at $g^\star$ in direction $h$
 yields
the desired assertion.\end{proof}

%The above result will allow us to write necessary optimality conditions for locally optimal shapes for different   instances for $\JJ$.
In the next two subsections we investigate the directional differentiability of shape functionals $\JJ(\O_g,y_{\O_g})$ w.r.t.\,the shape function $g$ for various types of instances for $\JJ$. Thanks to the above result, we can  provide necessary optimality conditions not only for globally but for locally optimal shapes as well.

\subsection{$H^2$ norms and pointwise observation on a given set}\label{sec:gobs}
In this subsection we are concerned with shape functionals of the type 
\[\JJ( y_\O):=  \frac{1}{2}\|y_\O -y_d \|_{H^2(E)}^2 +j_{obs}(y_\O) ,\quad E \subset \subset \O,\]
where $y_d \in H^2(E)$ and 
\[j_{obs}(y_\O) :=   \frac{1}{2}\sum_{k=1}^{N} (y_\O-y_d)^2(x^k_{obs}) ,\] where $N \in \N_+$, $x^k_{obs}\in \bar E, \ k = 1,...,N$. Note that performance indices such as $j_{obs}$ are of particular interest when the optimal shape must be chosen so that the state fits a desired value $y_d$ which can be measured only pointwise (at $x^k_{obs}$) by a finite number of sensors $N$.

%Again, to make sure that $y_\O(x^k_{obs}) $ exists, we ask that $g<0$ in $E$, which yields $E \subset \O_g$.

The corresponding reduced objective is
\[ j(g)=  \frac{1}{2}\|y_g -y_d \|_{H^2(E)}^2+\frac{1}{2}\sum_{k=1}^{N}(y_g-y_d)^2(x^k_{obs}),\quad g \in \FF \text{ with }g<0 \text{ in }\bar E.\]

We now intend to apply the differentiability result from Theorem \ref{thm:h1} in order to obtain a formula for the directional derivative of $j$.

\begin{proposition}\label{thm:1}Assume that $g, h \in \FF$ and \eqref{hx00} is true. Moreover, suppose that  $\O_g \in \OO_E$.  Then $j$ is directionally differentiable at $g$ in direction $h$ with  \[j'(g;h)= (y_g -y_d,q)_{H^2(E)}+j'_{obs}(g;h),\]where $q=\SS'(g;h)$ is the unique solution of \eqref{q}-\eqref{q0} and 
\[j_{obs}'(g;h)=\sum_{k=1}^{N} (y_g-y_d)(x^k_{obs})q(x^k_{obs}).\]
\end{proposition}\begin{proof}The assumption   $\O_g \in \OO_E$ implies $E \subset \subset \O_g$. From Theorem \ref{thm:h1} we deduce that $\SS:\FF \to H^2(E)$ is directionally differentiable with $q=\SS'(g;h)$, and  in particular 
\[\lim_{\l \searrow 0} q_\l(x^k_{obs})=q(x^k_{obs}),\]where we use $H^2(E) \embed \CC(\bar E)$ and we abbreviate again $q_\l:=\frac{y_{g+\l h}-y_g}{\l}$. Applying the chain rule then leads to the desired statement.\end{proof}

With the result in the above theorem at hand we can provide a first-order necessary  optimality condition in primal form for all locally optimal shapes of \eqref{p_sh}.

\begin{proposition}
Let  $\O^\star \in \OO_E$ be a locally  optimal shape of \eqref{p_sh}. Then, each of the  functions  $g^\star \in \FF_E$ that satisfy $\O_{g^\star}=\O^\star$ fulfill
\[ (y_{g^\star} -y_d,q)_{H^2(E)}+\sum_{k=1}^{N} (y_{g^\star}-y_d)(x^k_{obs})q(x^k_{obs}) \geq 0 \]for all $h \in \FF$ with \begin{equation*}
\{h=0\}\cap \{g^\star=0\}^\circ \neq \emptyset \quad \forall\,\{g^\star=0\}^\circ\subset \{g^\star=0\}, \end{equation*}where  $q \in W^{1,p}(\O_{g^\star})$  is the unique solution of \eqref{q}-\eqref{q0}. That is,  
\begin{equation} 
-\laplace  q +  \beta'(y_{g^\star};q) =0 \quad \text{in }  \O_{g^\star}, \end{equation}
\begin{equation}
q+ \nabla y_{g^\star}W =0 \quad \text{on } \partial \O_{g^\star},\end{equation}
and  $W:\partial \O_{g^\star} \to \R^2$ is given by Definition \ref{def:w}. 
\end{proposition}\begin{proof}
%We observe that $\JJ$ satisfies \eqref{j}, since $y_\o=y_\O$ in $\o$, for all $\o \in \OO, \O \in \CC_2$ with $\o \subset \O$. Hence, 
We   apply Theorem \ref{prop:necc}, according to which, it holds
\[j'(g^\star; h)\geq 0,\]provided that $j'(g^\star; h)$ exists.
The desired assertion is a consequence of Proposition \ref{thm:1}.
\end{proof}

\subsection{Distributed observation}\label{sec:do}
In this subsection, we are interested in the case 
\[\JJ(\O, y_\O):=   \int_{\O} J(y_\O(x))\,dx+\int_{\O}\psi(x)\,dx ,\]
where   $\psi \in \CC(\bar D)$ and  $J:\R \to \R$ is Lipschitz continuous on bounded sets and directionally differentiable in the same sense as $\beta$. In particular, we suppose in the sequel that   
  for every  $M > 0$ the estimate holds
  \begin{equation}\label{eq:Jlip}
   \|J(y_1) -J(y_2)\|_{L^r(\AA)} \leq L_{ J,M} \, \|y_1 - y_2\|_{L^r(\AA)} \quad \forall\, y_1,y_2 \in \clos{B_{L^\infty(\AA)}(0,M)},\ \forall\, 1\leq r \leq \infty,
  \end{equation}where    $L_{ J,M}> 0$. In addition,
    \begin{equation}\label{eq:dirJ}
   \Big\|\frac{J(y + \tau \,\dy) - J(y)}{\tau} - J'(y;\dy)\Big\|_{L^\varrho(\AA)} \stackrel{\tau \searrow 0}{\longrightarrow} 0 \quad \forall \, y,\dy \in L^\infty(\AA), \ \forall\, 1\leq \varrho<\infty.
  \end{equation}

%is Lipschitz continuous from $\CC(\bar \O)$ to $L^1(\O)$ while $\psi \in \CC(D)$.J satsifies the same assumptions as $\beta$. (even hoelder cont on bounded sets) 

Note that the above allows for tracking type functionals such as $J(y)=(y-y_d)^2,  y_d \in \R.$ Moreover, $J$ might be  non-smooth and   typical examples entering this category are
\[J(y)=\max\{0,y-y_d\},\quad J(y)=|y-y_d|, \qquad y_d \in \R.\]

In this subsection, the reduced objective associated to \eqref{p} is given by 
\[j(g)=    \int_{\{g<0\}} J(\SS(g)(x))dx+\underbrace{\int_{\{g<0\}}\psi(x)\,dx}_{=:\Psi(g)}, \quad g \in \FF\] and, thanks to Theorem \ref{thm:lr}, we have the following result.

%\ko{
%\begin{remark}$J(x,\SS(g)(x))$ if J is cont diff w.r.t. the second compoennt
%\end{remark}
%generally speaking
%\[\JJ'(g;h)=-\int_{g=0}J(\xi,y(\xi))\frac{h}{|\nabla g|} \,d\xi+ \int_{\O} \partial_y J(x,y_\O(x);q)dx \]}

%\ko{\begin{assumption}the Nemytskii op associated to $J:\R \to \R$  is lipschitz from $L^\varrho(M), \varrho<\infty$ to $L^1(M)$ for each M and directionally differentiable from $W_0^{1,p}(D)$ to $L^1(D)$.
%\end{assumption}\begin{remark}, e.g. $J(y)=|y|^{\varrho}$ smooth.modulus, maximum ; $J(y)=|y|^{1/\varrho}$ as well?\end{remark}}
\begin{proposition}\label{thm:2}Assume that $g, h \in \FF$ and \eqref{hx00} is true. Moreover, suppose that $\O_g$ is a domain.  Then $j$ is directionally differentiable at $g$ in direction $h$ with \[j'(g;h)= \int_{\{g<0\}}  J'(\SS(g)(x);q(x))dx-\int_{\{g=0\}} {J(0)}\frac{h}{|\nabla g|} \,d\xi+\Psi'(g;h),\]where $q$ is the unique solution of \eqref{q}-\eqref{q0} and \begin{equation}\label{psi}\Psi'(g;h)= -\int_{\{g=0\}}\psi \frac{h}{|\nabla g|}\,d \xi .\end{equation}\end{proposition}\begin{proof}
By using \cite[Cor.\,A.5]{1}, we obtain  
\[\frac{1}{\l}\Big\{\int_{\{g+\l h<0\}}\psi  \,dx- \int_{\{g <0\}}\psi \,dx\Big\}\to -\int_{\{g=0\}}\psi\frac{h}{|\nabla g|} \,d\xi \quad \text{as }\l \searrow 0,\]which yields \eqref{psi}.

We abbreviate for simplicity $y_{g+\l h}:=\SS(g+\l h)$, $\l>0$ small, and $y_g=\SS(g)$. Then, $g+\l h \in \FF$, cf.\,Lemma \ref{lem:glh}, and we have
\begin{align*}
& \frac{1}{\l}{[j(g+\l h)-j(g)]}- \int_{\O_g}J'(y_g(x);q(x))dx+\int_{g=0} J(\underbrace{y_g(\xi)}_{=0})\frac{h}{|\nabla g|} \,d\xi\\&=\int_{\O_{g+\l h} \setminus \O_g} \frac{ 1}{\l} J(y_{g+\l h})\,dx - \int_{\O_{g+\l h} \setminus \O_g} \frac{ 1}{\l} J(0)\,dx\\&\quad+\int_{\O_g} \frac{ 1}{\l}( J(y_{g+\l h})-J(y_g+\l q))dx
\\&\quad +\int_{\O_g} \frac{ 1}{\l}(J(y_g+\l q)-J(y_g))-  J'(y_g;q)dx
\\&\quad+\ { \int_{\O_{g+\l h} \setminus \O_g} \frac{ 1}{\l} J(0)\,d\xi+\int_{g=0} J(0)\frac{h}{|\nabla g|} \,d\xi} =:a_\l+b_\l +c_\l +d_\l.
\end{align*}
%\begin{align*}
%|a_\l|&\leq \frac{ 1}{\l} \int_{\O_{g+\l h} \setminus \O} |( J(y_{g+\l h}(x)-J(y(x)))|\,dx 
%\\&\leq \mu(\O_{g+\l h} \setminus \O)^{..}\frac{ 1}{\l} \|J(y_{g+\l h}) -J(y)\|_{L^\varrho(D)}
%\\&\leq \mu(\O_{g+\l h} \setminus \O)^{..}L_J \|q_\l\|_{\CC(D)}\to 0.
%\end{align*}
Before we analize each term let us observe that, for $\l>0$ small, it holds 
\[\|y_{g+\l h}\|_{\CC(D)}, \|y_g+\l q\|_{L^\infty(D)}\leq M,\]where $M>0$ is  independent of $\l$ (as a result of \eqref{w2p_bound} and since $q\in W^{1,p}(\O)$, $p>2$, is extended by zero outside $\O_g$, see Remark \ref{rem:q}).
From $y_g=0$ on $\O_{g+\l h} \setminus \O_g$ (see Remark \ref{rem:S}) we infer \begin{align*}
|a_\l|&\leq \frac{ 1}{\l} \int_{\O_{g+\l h} \setminus \O_g} |J(y_{g+\l h}(x))-J(y_g(x))|\,dx 
\\&\leq L_{J,M} \|q_\l\|_{L^1(\O_{g+\l h} \setminus \O_g)}
\\&\leq L_{J,M}  \mu(\O_{g+\l h} \setminus \O_g)^{\frac{\varrho-1}{\varrho}} \|q_\l\|_{L^\varrho(D)},\quad \varrho \in (1,\infty).
\end{align*}where in the second estimate we   used \eqref{eq:Jlip}. In light of \eqref{mu} and \eqref{ql_bound}, we arrive at   \[a_\l \to 0.\] Further, by employing again \eqref{eq:Jlip}  as well as Theorem \ref{thm:lr}, we deduce
\begin{align*}
|b_\l|&\leq \int_{\O_g} \frac{ 1}{\l} |J(y_{g+\l h}(x))-J(y_g+\l q))|dx
\\&\leq  L_{J,M} \Big\|\frac{y_{g+\l h}-y_g}{\l}- q\Big\|_{L^1(\O_g)}\to 0.
\end{align*}Due to  the directional differentiability of $J$, cf.\,\eqref{eq:dirJ}, we also have 
\begin{align*}
|c_\l|\leq \int_{\O_g}  \Big|\frac{ J(y_g+\l q)-J(y_g)}{\l}-J'(y_g;q) \Big|dx\to 0.
\end{align*}Finally, by using for instance \cite[Cor.\,A.5]{1}, we obtain 
$
d_\l
\to 0,
$
and by combining all the above convergences we arrive at the desired result.
\end{proof}

As in the previous subsection, we can now write down  a first-order necessary  optimality condition in primal form for all optimal shapes of \eqref{p_sh}.

\begin{proposition}
%Assume that $J:\R \to \R^+$ and $\psi \geq 0$ on $\bar D$.
Let  $\O^\star \in \OO_E$ be a locally  optimal shape of \eqref{p_sh}. Then, each of the  functions  $g^\star \in \FF_E$ that satisfy $\O_{g^\star}=\O^\star$ fulfill
\[   \int_{\{g^\star<0\}}  J'(\SS(g^\star);q)dx-\int_{\{g^\star=0\}} {(J(0)+\psi)}\frac{h}{|\nabla g^\star|} \,d\xi  \geq 0 \]for all $h \in \FF$ with \begin{equation*}
\{h=0\}\cap \{g^\star=0\}^\circ \neq \emptyset \quad \forall\,\{g^\star=0\}^\circ\subset \{g^\star=0\}.\end{equation*}Again,  $q \in W^{1,p}(\O_{g^\star})$  solves \begin{equation} 
-\laplace  q +  \beta'(y_{g^\star};q) =0 \quad \text{in }  \O_{g^\star}, \end{equation}
\begin{equation}
q+ \nabla y_{g^\star}W =0 \quad \text{on } \partial \O_{g^\star},\end{equation}
where  $W:\partial \O_{g^\star} \to \R^2$ is given by Definition \ref{def:w}. 
\end{proposition}\begin{proof}
%Again, we have $y_\o=y_\O$ in $\o$ for all $\o \in \OO, \O \in \CC_2$ with $\o \subset \O$. Thanks to the assumptions imposed on $J$ and $\psi$, it holds \[ \int_{\o} J(y_\o(x))\,dx+\int_{\o}\psi(x)\,dx \leq  \int_{\O} J(y_\O(x))\,dx+\int_{\O}\psi(x)\,dx, \]
%so that $\JJ$ satisfies \eqref{j}. 
The desired result   follows  by Theorem \ref{prop:necc} and  Proposition \ref{thm:2}.
\end{proof}

\section{The FVA and the speed method}\label{sec:sm}

This final section aims at putting into evidence   the strong connection between the  FVA and the prominent  speed method (also known as the velocity method) \cite{sz}. We present a rather formal approach which will lead us to the belief that    the vector field $W$ from Definition  \ref{def:w} generates the perturbations $\O_{g+\l h}$ of the original domain $\O_g$  in a similar way to the   speed method. This supports  our previous observation (Remark \ref{rem:sd}) that \eqref{q}-\eqref{q0} is the very   same PDE that characterizes   the shape derivative \cite[Eqs.\,(3.6),\,(3.8)]{sz}.

Throughout this section, $g,h \in \FF$ are fixed and \eqref{hx00} is satisfied. For simplicity, we assume that $\O_g$ is a (possibly multiply connected) domain. In all what follows, $\l_0>0$ is fixed, small enough, so that the unicity of the approximating closed curves is guaranteed, cf.\,Proposition \ref{lem:unique}.
Recall that  $W:\{g=0\}\to \R^2$ was defined in subsection \ref{sec:dd} as follows 
 \begin{equation}W(x):=w((z^{-1}(x))\quad \forall\, x \in \{g=0\}^\circ,\quad {\{g=0\}^\circ}\subset \{g=0\}\end{equation}where $w\in \CC^1([0,T]; \R^2)$ is the unique solution of \eqref{w} and $z$ describes the boundary of $\O_g$, that is, it solves \eqref{ham}. 
 %We also recall here \eqref{eq:wa}:
%\begin{equation}\label{eq:wa0}
%W(z(t))=w(t), \quad t \in [0,T_{\{g=0\}^\circ}],\quad \forall\,{\{g=0\}^\circ}\subset \{g=0\}.\end{equation}
%where $T_{\{g=0\}^\circ}>0$ is the period of ${\{g=0\}^\circ}$. 
For $\l \in [0,\l_0)$ we can analogously define  $\WW(\l,\cdot):\{g+\l h=0\}\to \R^2$ 
 \begin{equation}\label{wl}\WW(\l,x):= w(\l,(z_\l^{-1}(x))\quad \forall\, x \in \{g+\l h=0\}^\circ,\quad {\{g+\l h=0\}^\circ}\subset \{g+\l h=0\}\end{equation}where  \begin{equation}\label{wl0}  w({ \l},\cdot)=\lim_{\hat \l \to   \l} \frac{z_{ \hat \l}-z_{  \l}}{\hat \l-  \l}\in \CC^1([0,T]; \R^2) \end{equation}  is computed in a similar way to $w$ for $\l>0$ small enough, see the proof of \cite[Prop.\,3.2]{oc_t}. Hence, for fixed $t \in (0,\infty),$ $ w(\l,t)$ is the velocity vector  at $z_\l(t)$ of the trajectory $[0,\l_0) \ni \l \mapsto z_\l(t)$. Note that $\WW(0)=W$.
% Hence, $W$  associates to a point on $\partial \O_g$ its velocity when performing the variation of the domain from $\O_g$ to $\O_{g+\l h}$. 
%We keep in mind at all times that $W(\l,x)$ depends on   $h$ as well, cf.\,\eqref{w}.

 In the case of the speed method \cite[Sec.\,2.9]{sz}, the domain perturbations   take place via  $T_\l(\O_g)$, where $T_\l:D \to \R^2$ is the flow associated to a family of vector fields.
 %   \ko{recall W is lip cont(prove this, see commnedt part in the proof of lemma\ref{lem:W})}
Assuming  that there exists $\WW(\l,\cdot) \in C^{0,1}_c(D;\R^2)$  an extension of $\WW(\l,\cdot):\partial \O_{g+\l h} \to \R^2$ (denoted by the same symbol)  so that $\l \mapsto \WW(\l,\cdot) \in \CC([0,\l_0),C^{0,1}_c(D;\R^2))$, let $T_\l^{\WW}:D \to \R^2$ be its associated  flow, i.e., $T_\l^{\WW}(X)=x(\l,X)$ is the unique solution to 
 \begin{equation}\label{flow}
  \frac{d}{d\l} x(\l,X)={\WW}(\l,x(\l,X)) \quad \l \in [0,\l_0), \quad x(0,X)=X \in D.
   \end{equation}We are going to prove next that \begin{equation}\label{same} \partial \O_{g+\l h}=T_\l^{\WW}(\partial \O_g), \quad \quad \l \in [0,\l_0).\end{equation}
%\ko{ which
% means that \begin{equation}\label{same0}   \O_{g+\l h}={\text{Int}} T_\l^W(\partial \O_g), \quad \quad \l \in [0,\l_0),\end{equation}where $\mathcal{\text{Int}}$ is the interior of the curve $T_\l^W(\partial \O_g)$ 
%   Since w eare working in planar space, it suffices to show that 
   % \begin{equation}\label{same_b} \partial \O_{g+\l h}=\partial T_\l^W(\O_g), \quad \quad \l \in [0,\l_0).\end{equation}}
 To this end,  we introduce 
 another vector field  which   associates to each point on the boundary of   $\O_g$ a unique point on  the boundary of the perturbed domain $\O_{g+\l h}$.
For $\l \in [0,\l_0)$, this is given by the mapping $Z_{\l}:\{g=0\} \to \{g+\l h=0\}$ defined  as 
 \begin{equation}\label{zl}Z_{\l}(x):=z_\l (z^{-1}(x))\quad \forall\, x \in \{g=0\}^\circ,\quad {\{g=0\}^\circ}\subset \{g=0\}.\end{equation}

%\ro{by ift we have invertibility at the local level}
%This definition is to be understood as follows: Each $x \in \{g=0\}$ belongs to a certain component  $\{g=0\}^{\circ}$ (uniquely determined, as the compoennts of level sets are disjoint, see Prop). The trajectory of this periodic curve  is the solution of  the system \eqref{ham}. The value $((z) )^{-1}(x)$ computes the time point corresponding to $x$ on the associated trajectory with initial point $x_0^\alpha \in \R^2$. 
%This is uniquely determined in the intevral $[0,T)$)
%Note that the mapping $z :() \to $ is indeed invertible as its derivative is larger as zero in each time point on the trajectory of the associated closed curve.

Clearly, it holds 
\begin{equation}\label{zla}
Z_{\l}(z(t))=z_\l(t),\quad \forall\,t \in [0,T)\end{equation}and $Z_0(x)=x.$
For a point $X\in \partial \O_g$, we have with $t=z^{-1}(X)$ the following
\begin{align}
\frac{d}{d\l}Z_\l(X)&=\lim_{\hat \l \to   \l} \frac{Z_{ \hat \l}(X)-Z_{  \l}(X)}{\hat \l-  \l}
\\&=\lim_{\hat \l \to   \l} \frac{z_{ \hat \l}(t)-z_{  \l}(t)}{\hat \l-  \l}
\\&=w(\l,t)={\WW}(\l,Z_\l(X)) \quad \l \in [0,\l_0),\end{align} where we use \eqref{zla}, \eqref{wl} and \eqref{wl0}.
Going back to \eqref{flow}, we see that \[x(\l,X)=Z_\l(X) \quad \text{for }X \in \partial \O_{g},\] since \eqref{flow} is uniquely solvable.

{Therefore, the points on the perturbed boundary $\partial \O_{g+\l h}$ appear as flows of $\WW$ with initial condition on the original boundary.
We underline that   our $T_\l^{\WW}$ is well-defined only on $\partial \O_g$ (not in the whole $D$ as in \cite{sz}); {nevertheless,   the entire domain is automatically perturbed, being the interior of a finite union of disjoint planar closed curves without self intersections (Proposition \ref{prop:ham}). }

Thus, we now have an additional (however rather formal) confirmation that the limit  \eqref{diff_s} we calculated in subsection \ref{sec:dd} (Theorems \ref{thm:h1} and \ref{thm:lr})  is in fact 
\[\lim_{\l \searrow 0}\frac{y_{T_\l^{\WW}(\O_g)}-y_{\O_g}}{\l},\]that is, the shape derivative of $y_{\O_g}$ in direction  $\WW$.

\section*{Acknowledgment}
This work was supported by the DFG grant BE 7178/3-1 for the project "Optimal Control of Viscous
Fatigue Damage Models for Brittle Materials: Optimality Systems".  

\bibliographystyle{plain}
\bibliography{Betz}

\begin{thebibliography}{10}

\bibitem{sign_shape0}
S.~Adly, L.~Bourdin, F.~Caubet, and A.~Jacob~de Cordemoy.
\newblock Shape optimization for variational inequalities: The scalar {T}resca
  friction problem.
\newblock {\em SIAM Journal on Optimization}, 33(4):2512--2541, 2023.

\bibitem{allaire}
G.~Allaire.
\newblock {\em Shape optimization by the homogenization method}.
\newblock Springer, New York, 2001.

\bibitem{1}
S.~Amstutz.
\newblock Analysis of a level set method for topology optimization.
\newblock {\em Optim. Methods Softw.}, 26(4-5):555--573, 2011.

\bibitem{p1}
L.~Betz.
\newblock Approximation of shape optimization problems with non-smooth {PDE}
  constraints: an optimality conditions point of view.
\newblock {\em SIAM J. Control and Optimiz.}, 63(3), 2025.

\bibitem{p3}
L.~Betz.
\newblock Necessary conditions for the optimal control of a shape optimization
  problem with non-smooth {PDE} constraints.
\newblock {\em SIAM J. Control and Optimiz.}, 63(5), 2025.

\bibitem{sign_shape}
A.~J. de~Cordemoy.
\newblock Shape optimization for contact problem involving {S}ignorini
  unilateral conditions.
\newblock {\em Mathematical Control and Related Fields}, 2025.

\bibitem{dz}
M.~Delfour and J.-P. Zolesio.
\newblock {\em Shapes and Geometries, 2nd ed.}
\newblock SIAM, Philadelphia, 2001.

\bibitem{evans}
L.C. Evans.
\newblock {\em Partial Differential Equations}.
\newblock American Mathematical Society, 2010.

\bibitem{gt}
D.~Gilbarg and N.~S. Trudinger.
\newblock {\em Elliptic Partial Differential Equations of Second Order}.
\newblock Springer Verlag, Berlin, 2001.

\bibitem{hmt}
A.~Halanay, C.~Murea, and D.~Tiba.
\newblock Some properties of the period for certain ordinary differential
  systems and applications to topology optimization of variational
  inequalities.
\newblock {\em Banach Center Publications}, 127:129--146, 2024.

\bibitem{hs}
C.~Heinemann and K.~Sturm.
\newblock Shape optimization for a class of semilinear variational inequalities
  with applications to damage models.
\newblock {\em SIAM J. Math. Anal.}, 48(5):3579--3617, 2016.

\bibitem{hp}
A.~Henrot and M.~Pierre.
\newblock {\em Shape variation and optimization. A geometrical analysis}.
\newblock European Mathematical Society, 2018.

\bibitem{hl}
M.~Hinterm\"{u}ller and A.~Laurain.
\newblock Optimal shape design subject to elliptic variational inequalities.
\newblock {\em SIAM Journal on Control and Optimization}, 49(3):1015--1047,
  2011.

\bibitem{hsd}
M.~W. Hirsch, S.~Smale, and R.~L. Devaney.
\newblock {\em Differential Equations, Dynamical Systems, and an Introduction
  to Chaos}.
\newblock Academic Press, 2012.

\bibitem{kun}
K.~Ito, K.~Kunisch, and G.H. Peichl.
\newblock Variational approach to shape derivatives.
\newblock {\em ESAIM: Control, Optimisation and Calculus of Variations},
  14(3):517--539, 2008.

\bibitem{kk}
H.~Kasumba and K.~Kunisch.
\newblock On computation of the shape {H}essian of the cost functional without
  shape sensitivity of the state variable.
\newblock {\em J. Optim. Theory Appl.}, 162:779--804, 2014.

\bibitem{kov_k}
V.~Kovtunenko and K.~Kunisch.
\newblock Shape derivative for penalty-constrained nonsmooth--nonconvex
  optimization: Cohesive crack problem.
\newblock {\em Journal of Optimization Theory and Applications},
  194(2):597--635, 2022.

\bibitem{kov_k0}
V.~A. Kovtunenko and K.~Kunisch.
\newblock Directional differentiability for shape optimization with variational
  inequalities as constraints.
\newblock {\em ESAIM Control Optim. Calc. Var.}, 29:Paper No. 64, 30, 2023.

\bibitem{2}
A.~Laurain.
\newblock Analyzing smooth and singular domain perturbations in level set
  methods.
\newblock {\em SIAM Journal on Mathematical Analysis}, 50(4):4327--4370, 2018.

\bibitem{mnt}
R.~Makinen, P.~Neittaanmaki, and D.~Tiba.
\newblock On a fixed domain approach for a shape optimization problem.
\newblock In {\em Ames, W.F., van der Houwen, P.J. (eds) Computational and
  Applied Mathematics II: Differential Equations}, pages 317--326.
  North-Holland, Amsterdam, 1992.

\bibitem{to}
C.~Murea and D.~Tiba.
\newblock Topological optimization via cost penalization.
\newblock {\em Topological Methods in Nonlinear Analysis}, 54(2B):1023--1050,
  2019.

\bibitem{tm_lin}
C.~Murea and D.~Tiba.
\newblock Topological optimization and minimal compliance in linear elasticity.
\newblock {\em Evol. Equ. Control Theory}, 9(4):1115--1131, 2020.

\bibitem{t_jde}
C.~Murea and D.~Tiba.
\newblock Periodic {H}amiltonian systems in shape optimization problems with
  {N}eumann boundary conditions.
\newblock {\em Journal of Differential Equations}, 321:1--39, 2022.

\bibitem{mt_new2}
C.~Murea and D.~Tiba.
\newblock Penalization of stationary {N}avier-{S}tokes equations and
  applications in topology optimization.
\newblock {\em Japan Journal of Industrial and Applied Mathematics},
  42(1):289--329, 2025.

\bibitem{kirchhoff}
C.M. Murea.
\newblock Topology optimization problems for clamped {K}irchhoff-{L}ove plates.
\newblock {\em Ann. Acad. Rom. Sci. Ser. Math. Appl.}, 17(1):193--221, 2025.

\bibitem{plate}
C.M. Murea and D.~Tiba.
\newblock Optimization of a plate with holes.
\newblock {\em Comput. Math. Appl.}, 77:3010--3020, 2019.

\bibitem{stokes}
C.M. Murea and D.~Tiba.
\newblock Topology optimization for the {S}tokes system.
\newblock {\em Math. Rep.}, 24(1-2):301--317, 2022.

\bibitem{nsz}
P.~Neittaanmaki, J.~Soko{l}owski, and J.P. Zolesio.
\newblock Optimization of the domain in elliptic variational inequalities.
\newblock {\em Appl. Math. and Optimiz.}, 18(1):85--98, 1998.

\bibitem{nst_book}
P.~Neittaanm\"{a}ki, J.~Sprekels, and D.~Tiba.
\newblock {\em Optimization of Elliptic Systems. Theory and Applications}.
\newblock Springer, Berlin, 2006.

\bibitem{top_book}
A.~A. Novotny and J.~Soko{l}owski.
\newblock {\em Topological Derivatives in Shape Optimization}.
\newblock Springer, Berlin, 2013.

\bibitem{os}
S.~Osher and J.A. Sethian.
\newblock Fronts propagating with curvature-dependent speed: algorithms based
  on {H}amilton-{J}acobi formulations.
\newblock {\em J. Comput. Phys.}, 79(1):12--49, 1988.

\bibitem{pir}
O.~Pironneau.
\newblock {\em Optimal Shape Design for Elliptic Systems}.
\newblock Springer Verlag, Berlin, 1984.

\bibitem{pressley}
A.~Pressley.
\newblock {\em Elementary Differential Geometry, 2nd ed}.
\newblock Springer, London, 2012.

\bibitem{sok}
J.~Soko{l}owski.
\newblock Shape sensitivity analysis of nonsmooth variational problems.
\newblock {\em Boundary Control and Boundary Variations. Proceedings of the
  IFIP WG 7.2 Conference, Lecture Notes in Control and Information Sciences,
  Springer Verlag, Berlin}, 100:265--285, 1988.

\bibitem{zs}
J.~Soko{l}owski and A.~{Z}ochowski.
\newblock Optimality conditions for simultaneous topology and shape
  optimization.
\newblock {\em SICON}, 42(4):1198--1221, 2003.

\bibitem{sz}
J.~Soko{l}owski and J.-P. Zolesio.
\newblock {\em Introduction to Shape Optimization. {S}hape Sensitivity
  Analysis}.
\newblock Springer, Berlin, 1992.

\bibitem{k_sturm}
K.~Sturm.
\newblock Minimax {L}agrangian approach to the differentiability of nonlinear
  {PDE} constrained shape functions without saddle point assumption.
\newblock {\em SICON}, 53(4):2017--2039, 2015.

\bibitem{it_h}
D.~Tiba.
\newblock Iterated {H}amiltonian type systems and applications.
\newblock {\em J. Differential Equations}, 264(8):5465--5479, 2018.

\bibitem{pen}
D.~Tiba.
\newblock A penalization approach in shape optimization.
\newblock {\em Atti Accad. Pelorit. Pericol. Cl. Sci. Fis. Mat. Nat.}, 96(1),
  2018.

\bibitem{oc_t}
D.~Tiba.
\newblock Optimality conditions and {L}agrange multipliers for shape and
  topology optimization problems.
\newblock {\em Atti Accad. Pelorit. Pericol. Cl. Sci. Fis. Mat. Nat.}, 101(1),
  2023.

\bibitem{troe}
F.~Tr{\"o}ltzsch.
\newblock {\em Optimal Control of Partial Differential Equations}, volume 112
  of {\em Graduate Studies in Mathematics}.
\newblock American Mathematical Society, Providence, 2010.
\newblock Theory, methods and applications, Translated from the 2005 German
  original by J{\"u}rgen Sprekels.

\bibitem{3}
N.~P. van Dijk, K.~Maute, M.~Langelaar, and F.~van Keulen.
\newblock Level-set methods for structural topology optimization: a review.
\newblock {\em Struct. Multidiscip. Optim.}, 48(3):437--472, 2013.

\bibitem{dw_gw}
G.~Wachsmuth and D.~Walter.
\newblock No-gap second-order conditions for minimization problems in spaces of
  measures.
\newblock {\em arXiv}, https://arxiv.org/abs/2403.12001, 2024.

\bibitem{chen_wu}
C.~Ya-Zhe and W.~Lan-Cheng.
\newblock {\em Second Order Elliptic Equations and Elliptic Systems}.
\newblock American Mathematical Society, 1998.

\end{thebibliography}

\end{document}